\documentclass[11pt]{amsart}

\usepackage{amssymb}
\usepackage{amsmath}
\usepackage{amsthm}
\usepackage[dvipsnames]{xcolor}
\usepackage[pdftex]{graphicx} 
\usepackage{tikz}
\usepackage{enumerate}
\usepackage{multirow}
\usepackage{url}
\usepackage{genyoungtabtikz}

\newtheorem{theorem}{Theorem}[section]
\newtheorem{lemma}[theorem]{Lemma}
\newtheorem{corollary}[theorem]{Corollary}
\newtheorem{proposition}[theorem]{Proposition}

\theoremstyle{definition}
\newtheorem{example}[theorem]{Example}
\newtheorem{remark}[theorem]{Remark}
\newtheorem{definition}[theorem]{Definition}

\numberwithin{equation}{section}

\newcommand{\HH}{\mathcal{H}}
\newcommand{\OO}{\mathcal{O}}
\newcommand{\KK}{\mathbb{K}}
\newcommand{\II}{\mathcal{I}}
\newcommand{\interval}[2]{[\![#1,#2]\!]}

\definecolor{light-gray}{gray}{0.6}

\newcommand{\bsq}[1]{
\draw[fill=black] #1 rectangle +(1,1);
\draw[<-,light-gray,thick] #1 ++(0,0.5) -- ++(1,0);
\draw[<-,light-gray,thick] #1 ++(0.5,1) -- ++(0,-1);
}
\newcommand{\wsq}[1]{
\draw #1 rectangle +(1,1);
\draw[->,rounded corners=0.35cm,thick,light-gray] #1 ++ (0.5,0) -- ++(0,0.5) -- ++(-0.5,0);
\draw[->,rounded corners=0.35cm,thick,light-gray] #1 ++(1,0.5) -- ++(-0.5,0) -- ++(0,0.5);
}

\newcommand\ybl{\Yfillcolour{black}}
\newcommand\ywh{\Yfillcolour{white}}

\begin{document}
\title[Grassmann Necklaces and Restricted Permutations]{From Grassmann necklaces to restricted permutations and back again}
\author{Karel Casteels}
\address{Department of Mathematics \\ University of California, Santa Barbara \\ Santa, Barbara, California, 93106}
\email{casteels@ucsb.edu}

\author{Si\^an Fryer}
\address{School of Mathematics \\ University of Leeds \\ Leeds LS2 9JT\\ UK.}
\email{s.fryer@leeds.ac.uk}

\begin{abstract}
We study the commutative algebras $Z_{JK}$ appearing in Brown and Goodearl's extension of the $\HH$-stratification framework, and show that if $A$ is the single parameter quantized coordinate ring of $M_{m,n}$, $GL_n$ or $SL_n$, then the algebras $Z_{JK}$ can always be constructed in terms of centres of localizations. The main purpose of the $Z_{JK}$ is to study the structure of the topological space $spec(A)$, which remains unknown for all but a few low-dimensional examples.  We explicitly construct the required denominator sets using two different techniques (restricted permutations and Grassmann necklaces) and show that we obtain the same sets in both cases.  As a corollary, we obtain a simple formula for the Grassmann necklace associated to a cell of totally nonnegative real $m\times n$ matrices in terms of its restricted permutation.
\end{abstract}
\maketitle

\section{Introduction}\label{s:introduction}

Let $A$ be a quantized coordinate ring (or ``quantum algebra'') with parameter $q$ and equip its prime spectrum $spec(A)$ with the Zariski topology.  A major open question in the study of quantum algebras is understanding the topological structure of this space, and relating it to the corresponding space of Poisson-prime ideals in the semi-classical limit of $A$.

We fix a field $\KK$ of arbitrary characteristic and $q \in \KK^{\times}$ not a root of unity.  All algebras considered in this paper will be $\KK$-algebras.

The key technique in the study of $spec(A)$ is that of $\HH$-stratification, where $\HH$ is a torus acting rationally on $A$.  Write $\HH$-$spec(A)$ for the set of \textit{$\HH$-primes}: the prime ideals of $A$ which are also invariant under the action of $\HH$.  The $\HH$-primes induce a stratification of $spec(A)$, where for each $J \in \HH$-$spec(A)$ the associated stratum is
\begin{equation}\label{eq:def of strata}spec_J(A) := \{P \in spec(A) : J \text{ is the largest }\HH\text{-prime contained in }P\}.\end{equation}
This stratification has many nice properties. In particular, each $spec_J(A)$ is homeomorphic to the prime spectrum of a commutative Laurent polynomial ring $Z(A_J)$, which can be constructed explicitly as the centre of a localization of $A/J$ \cite[Theorem~II.2.13]{BGbook}.  These algebras $Z(A_J)$ are now reasonably well understood, e.g. \cite{BCL,Yakimov2}.

This approach gives us an excellent picture of the individual strata, but it comes at the cost of losing information about the interactions between primes from different strata.  Being able to identify inclusions of prime ideals from different strata is crucial for constructing a complete picture of the topological structure of $spec(A)$, but this information is often surprisingly difficult to obtain: for example, in the setting of quantum matrices only the prime spectra of $\OO_q(SL_2)$, $\OO_q(GL_2)$, and $\OO_q(SL_3)$ (\cite{HL}, \cite{BGconjecture}, \cite{Fryer} respectively) have been fully described.

In order to keep track of this extra information, Brown and Goodearl introduced an extension of the $\HH$-stratification framework in \cite{BGconjecture}.  A key part of this framework is a collection of commutative algebras $Z_{JK}$ (one for each pair of $\HH$-primes $J \subsetneq K$), which form bridges between the Laurent polynomial rings $Z(A_J)$, $Z(A_K)$ associated to the strata of $J$ and $K$ respectively.  More precisely, Brown and Goodearl define $\KK$-algebra homomorphisms
\[Z(A_J) \stackrel{g_{JK}}{\longleftarrow} Z_{JK} \stackrel{f_{JK}}{\longrightarrow} Z(A_K)\]
and conjecture that the comorphisms $g_{JK}^{\circ}: spec(Z(A_J)) \rightarrow spec(Z_{JK})$ and $f_{JK}^{\circ}: spec(Z(A_K)) \rightarrow spec(Z_{JK})$ can be used to construct a map
\[\varphi_{JK}: spec_J(A) \longrightarrow spec_K(A)\]
which encodes the missing information about inclusions of primes between these two strata (see \cite[\S3]{BGconjecture}).  Thus the structure of $spec(A)$ could be described in terms of the strata \eqref{eq:def of strata} and finitely many maps between them.

Ideally these $Z_{JK}$ would be constructed in terms of centres of localizations, by analogy to the algebras $Z(A_J)$; the definitions of $g_{JK}$ and $f_{JK}$ would then follow almost immediately from universal properties.  But localization can be a messy business when $A$ is noncommutative: a multiplicative set $E \subset A$ must satisfy certain conditions (the \textit{Ore} and \textit{denominator} conditions, see e.g. \cite[Chapter 10]{GW1}) to ensure that $A[E^{-1}]$ is even well-defined.  As a result, the algebras $Z_{JK}$ and the maps $g_{JK},f_{JK}$ are defined in \cite{BGconjecture} without reference to localization: this guarantees their existence but makes them difficult to work with.

In this paper we show that the $Z_{JK}$ can be realised as centres of localizations when $A$ is any one of $\OO_q(M_{m,n})$, $\OO_q(GL_n)$, or $\OO_q(SL_n)$.  We do this by constructing explicit finitely generated denominator sets for each pair of $\HH$-primes $J \subsetneq K$: this is the content of Theorem~\ref{res:GN gens Ore set, intro version} below.  We achieve this in the first instance by exploiting the connection between $\HH$-primes in quantum matrices and cells of totally nonnegative real matrices from \cite{GLL2,GLL1}: for background and definitions, see \S\ref{ss:background_tnn}. Grassmann necklaces are defined in Definition~\ref{def:grassmann necklace}, and the formal definition of the algebras $Z_{JK}$ discussed above is given in equation \eqref{eq:first def of ZJK}.

\begin{theorem}\label{res:GN gens Ore set, intro version} (Theorem \ref{res:GN constructs the Ore set we need}) 
Let $K$ be a $\HH$-prime in $\OO_q(M_{m,n})$, $S_K$ the totally nonnegative cell in $M_{m,n}^{tnn}$ defined by the same set of minors as $K$, and $\II_K$ the Grassmann necklace of $S_K$.  Having identified quantum minors and Pl\"ucker coordinates as in equation \eqref{eq:translate minors and plucker coordinates}, let $E_K$ be the multiplicative set in $\OO_q(M_{m,n})$ generated by $\II_K$.  Then:
\begin{enumerate}
\item $E_K \cap K = \emptyset$, and $E_K$ is a denominator set in $\OO_q(M_{m,n})$.
\item Fix $J \in \HH$-$spec(\OO_q(M_{m,n}))$ with $J \subsetneq K$, and write $E_{JK}$ for the projection of $E_K$ to the quotient $\OO_q(M_{m,n})/J$.  Then we have the following equality of algebras:
\[Z_{JK} = Z\big(\OO_q(M_{m,n})/J[E_{JK}^{-1}]\big).\]
\end{enumerate}
\end{theorem}
We extend this result to $\OO_q(GL_n)$ and $\OO_q(SL_n)$ in Corollary~\ref{res:gln} and Corollary~\ref{res:sln} respectively.

Theorem~\ref{res:GN gens Ore set, intro version} paves the way for a comprehensive study of the algebras $Z_{JK}$ for quantum matrices.  As demonstrated in \cite{BGconjecture,Fryer}, it is expected that these algebras can be used to obtain a complete picture of the topological structure of $spec(A)$ (modulo a technical condition, which we do not discuss here).

In \cite{FryerYakimov}, the second author and Yakimov studied the question of realising the algebras $Z_{JK}$ in terms of localizations for a much wider class of algebras, using the language of quantum groups and Demazure modules.  The denominator sets in \cite[Main Theorem]{FryerYakimov} are given in terms of quasi $\mathcal{R}$-matrices, however, which do not lend themselves to easy computation.  In \S\ref{s:permutation version}, we prove the following theorem:
\begin{theorem}\label{res: denominator sets agree, intro version} 
(Theorem \ref{res:yakimov's minors come from the chains rooted at squares})
The generators for the denominator sets in \cite[Main Theorem]{FryerYakimov}, restricted to the case $\OO_q(M_{m,n})$, agree with those in Theorem~\ref{res:GN gens Ore set, intro version} (up to scalars).
\end{theorem}
This gives a concrete interpretation of the results of \cite{FryerYakimov} for the case of quantum matrices.  Theorem~\ref{res: denominator sets agree, intro version} is achieved by expressing the generators in \cite[Main Theorem]{FryerYakimov} in terms of \textit{restricted permutations}, i.e. the set of permutations
\[\mathcal{S} = \big\{\sigma \in S_{m+n} : -n \leq \sigma(i) - i \leq m,\ \forall i \in \{1, \dots, m+n\}\big\},\]
and then relating this to the Grassmann necklaces in Theorem~\ref{res:GN gens Ore set, intro version} via a graphical method introduced by Oh in \cite{Oh}.  We obtain the following description of Grassmann necklaces in terms of restricted permutations as an easy corollary of this.
\begin{theorem}\label{res:computing grassmann necklace in terms of permutations, intro} (Theorem~\ref{res:gn in terms of permutation})
Let $K\in \HH$-$spec(\OO_q(M_{m,n}))$, and let $v$ be its associated restricted permutation. Write $\interval{a}{b}$ for the set of integers $\{i \in \mathbb{Z}: a \leq i \leq b\}$, and define $\omega$ to be the permutation in $S_{m+n}$ given by
\[\omega = \begin{pmatrix}1 & 2 & \dots & m & m+1 & \dots m+n\\ m & m-1 & \dots & 1 & m+n & \dots m+1 \end{pmatrix}\]  
Then the Grassmann necklace $\II_K = (I_1, \dots, I_{m+n})$ associated to $K$ is given in terms of $v$ by
\[I_{m+n-k+1} = \left\{\begin{array}{cc} \Big(\interval{1}{m}\cup \interval{m+n-k+1}{m+n}\Big) \backslash \omega v \interval{1}{k} & \qquad k \in \interval{1}{n}, \\
\Big(\interval{1}{m+n} \backslash \omega v\interval{1}{k} \Big) \cup \interval{m+n-k+1}{m} & \qquad k \in \interval{n+1}{m+n}.
\end{array}\right.\]
\end{theorem}

The structure of this paper is as follows.  In \S\ref{s:background, notation} we introduce the required notation and background information, including details of the remarkable connection between the study of $\HH$-primes and total nonnegativity obtained in \cite{GLL2,GLL1}.  In \S\ref{s:grassmann necklaces} we construct our denominator sets $E_{JK}$ using Grassmann necklaces and the language of total nonnegativity, and hence prove Theorem~\ref{res:GN gens Ore set, intro version}.  We describe Oh's ``chain rooted at a square'' method for reading off the Grassmann necklace (and hence the set $E_{JK}$) directly from the Cauchon diagram associated to $K$, and give a careful proof of its properties.

In \S\ref{s:permutation version} we consider the question of constructing denominator sets from a ring theory and representation theory perspective, and use the chain construction from \S\ref{s:grassmann necklaces} to prove Theorem~\ref{res: denominator sets agree, intro version}, i.e. that the denominator sets constructed in \cite{FryerYakimov} agree with those constructed via Grassmann necklaces in Theorem~\ref{res:GN gens Ore set, intro version}.  This allows us to describe the relationship between Grassmann necklaces and restricted permutations in this setting (Theorem~\ref{res:computing grassmann necklace in terms of permutations, intro}).  Finally, in \S\ref{s:special cases} we extend the results of \S\ref{s:grassmann necklaces} and \S\ref{s:permutation version} to $\OO_q(GL_n)$ and $\OO_q(SL_n)$, and show that our denominator sets can also be used to study the algebras $Z(A_K)$ associated to the individual strata $spec_K(A)$.  This will provide a unified approach for future work studying the various localizations appearing in the results of \cite{BGconjecture}, and hence a greater understanding of the prime spectra of these algebras.

\section{Background and Notation}\label{s:background, notation}

\subsection{Notation}\label{ss:background notation}
Throughout, we fix $\KK$ to be an infinite field of arbitrary characteristic and $q \in \KK^{\times}$ not a root of unity.  Note that the representation theory of quantum algebras is completely different when $q$ is a root of unity, so this assumption cannot be relaxed.  All algebras considered will be $\KK$-algebras.

For $t \in \mathbb{N}$, let $S_t$ denote the symmetric group on $t$ elements.  \textbf{All permutations will be composed from right to left}.  We will often fix two integers $m,n \geq 2$ and work with elements of the symmetric group $S_{m+n}$; in this case, we write $w_m^o$ for the longest word in the subgroup $S_m$, i.e. the permutation $\left(\begin{smallmatrix}1 & 2 & \dots & m \\ m & m-1 & \dots &1 \end{smallmatrix}\right)$. Similarly, let $w_{m,n}^{o}$ denote the longest word in $S_n$, where now $S_n$ is viewed as the subgroup of permutations on the set $\{m+1, \dots, m+n\}$.  Finally, $c$ will denote the permutation $\left(\begin{smallmatrix} 1 & 2 & \dots & m+n \\ 2 & 3 & \dots & 1 \end{smallmatrix}\right) \in S_{m+n}$.

Let $\interval{a}{b}$ denote the interval $\{i \in \mathbb{Z}: a \leq i \leq b\}$.  When $1 \leq a < b$, we write $\binom{[b]}{a}$ for the collection of all subsets $I \subset \interval{1}{b}$ with $|I| = a$.  

For fixed integers $m,n$ as above, let $I \subset\interval{1}{m+n}$ and define the projections $p_1(I)$ and $p_2(I)$ as follows: 
\[p_1(I) = I \cap \interval{1}{m}, \qquad p_2(I) = I \cap \interval{m+1}{m+n}.\]
We will also write $J + m$ as shorthand for the set $\{j + m : j \in J\}$.  

\subsection{The quantized coordinate ring of $m\times n$ matrices}\label{ss:background_qm}

The algebra of quantum matrices $\OO_q(M_{m,n})$ is defined to be the algebra generated by $mn$ variables $\{X_{ij}: 1 \leq i \leq m,\ 1 \leq j \leq n\}$ subject to the relations
\begin{align*}
X_{ij}X_{il} &= qX_{il}X_{ij} \\
X_{ij}X_{kj} &= qX_{kj}X_{ij} \\
X_{il}X_{kj}  &= X_{kj}X_{il} \\
X_{ij}X_{kl} - X_{kl}X_{ij} &= (q-q^{-1})X_{il}X_{kj}
\end{align*}
for all $i < k$ and $j < l$.  

If $S = \{s_1 < s_2 < \dots < s_r\} \subset \interval{1}{m}$ and $T = \{t_1< t_2 < \dots < t_r\} \subset\interval{1}{n}$ are two sets with $|S| = |T|$, then the \textit{quantum minor} $[S|T]_q$ indexed by these sets is
\begin{equation}\label{eq:def of quantum minor}
[S|T]_q = \sum_{\sigma \in S_r} (-q)^{l(\sigma)} X_{s_1,\sigma(t_1)}\dots X_{s_r,\sigma(t_r)},
\end{equation}
where $l(\sigma)$ denotes the length of the permutation $\sigma \in S_r$ (i.e. the number of inversions in $\sigma$).

Define an action of the torus $\HH = (\KK^{\times})^{m+n}$ on $\OO_q(M_{m,n})$ by
\begin{equation}\label{eq:H action def}h\cdot X_{ij} :=  \alpha_i \beta_j X_{ij}, \quad h = (\alpha_1,\dots,\alpha_m,\beta_1,\dots,\beta_n) \in \HH,\end{equation}
and extend it linearly to the whole of $\OO_q(M_{m,n})$.  This is a rational action of $\HH$ on $\OO_q(M_{m,n})$ (e.g. \cite[II.1.15]{BGbook}).  

If $P$ is a prime ideal of $\OO_q(M_{m,n})$ which is invariant under the action of $\HH$ (i.e. $h(P) = P$ for all $h \in \HH$), we call $P$ a $\HH$\textit{-prime}.  Let $\HH$-$spec(\OO_q(M_{m,n}))$ denote the set of all $\HH$-primes in $\OO_q(M_{m,n})$; when the $\HH$-action is as in \eqref{eq:H action def} (which it always will be in this paper), $\HH$-$spec(\OO_q(M_{m,n}))$ is a finite set.

The $\HH$-primes are used to study the prime spectrum of $\OO_q(M_{m,n})$, but are also interesting objects of study in their own right, e.g. \cite{Casteels2,Cauchon,GL2,Launois,Yakimov1}.  A key tool in many of these papers is \textit{Cauchon diagrams}, which we now define.

\begin{definition}\label{def:cauchon property}
Let $C$ be a Young diagram together with a choice of colouring (black or white) for each square.  Then $C$ is a Cauchon diagram if no black square in $C$ has a white square both to its left in the same row and above it in the same column.  See Figure~\ref{fig:examples and non-examples of Cauchon diagrams} for examples and non-examples of Cauchon diagrams.
\end{definition}

In \cite{Cauchon}, Cauchon proved that $\HH$-$spec(\OO_q(M_{m,n}))$ is in bijection with the set of rectangular Cauchon diagrams with $m$ rows and $n$ columns.

Cauchon diagrams have also appeared independently in the work of Postnikov \cite{Postnikov} under the name $\reflectbox{L}$-diagrams (or Le-diagrams), where they are used to parametrise cells of totally nonnegative matrices.  One can easily pass between the two definitions by replacing all white squares with a $+$ symbol, and all black squares with a $0$.

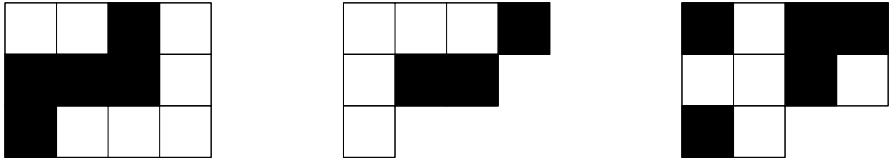
\begin{figure}

\begin{tikzpicture}[scale=1.5]
\tyng(0cm,0cm,4,4,4)
\tgyoung(0cm,0cm,!\ybl::;:,;;;:,;:::!\ywh)

\tyng(3cm,0cm,4,3,1)
\tgyoung(3cm,0cm,!\ybl:::;,:;;,:!\ywh)

\tyng(6cm,0cm,4,4,2)
\tgyoung(6cm,0cm,!\ybl;:;;,::;:,;:)
\end{tikzpicture}

\caption{The left and right diagrams are Cauchon diagrams; the middle one is not.}
\label{fig:examples and non-examples of Cauchon diagrams}
\end{figure}

\subsection{Total nonnegativity in real matrices and the real Grassmannian}\label{ss:background_tnn}
A real $m \times n$ matrix is called \textit{totally nonnegative} (TNN) if all of its minors are zero or positive.  More generally, a point in the real Grassmannian $Gr(k,d)$ is totally nonnegative if it can be represented by a $k \times d$ matrix whose Pl\"ucker coordinates (i.e. its $k\times k$ or \textit{maximal} minors) are all nonnegative.  We write $M_{m,n}^{tnn}$ and $Gr(k,d)^{tnn}$ for the spaces of totally nonnegative matrices and totally nonnegative points in the Grassmannian respectively.

We identify the Pl\"ucker coordinates in $Gr(k,d)$ with the set $\binom{[d]}{k}$ defined in \S\ref{ss:background notation}.  For $I \in \binom{[d]}{k}$ and $M \in Gr(k,d)$, write $\Delta_I(M)$ for the $k\times k$ minor of $M$ whose columns are indexed by $I$.

We can think of points in $Gr(k,d)$ as equivalence classes of $k\times d$ real matrices of rank $k$, where the equivalence is given by row operations.  An obvious choice of representative for a given equivalence class is therefore the reduced row echelon form (RREF) of the matrices in that class.  This allows us to define an embedding $M_{m,n} \hookrightarrow Gr(m,m+n)$ by identifying $m\times n$ matrices with the subspace of points in $Gr(m,m+n)$ whose RREF has the $m\times m$ identity matrix in the first $m$ columns.  We fix a specific embedding $\theta$: 
\begin{equation*}\begin{gathered}
\\
\begin{pmatrix}
b_{1,1} & \dots & b_{1,n} \\
b_{2,1} & \dots & b_{2,n} \\
\vdots & \ddots & \vdots \\
b_{m-1,1} & \dots & b_{m-1,n}\\
b_{m,1} & \dots & b_{m,n}
\end{pmatrix}
 \stackrel{\theta}{\mapsto }
\begin{pmatrix}
1 & 0 & \dots & 0 & 0 & (-1)^{m-1}b_{1,n} & \dots & (-1)^{m-1}b_{1,1} \\
0 & 1 & \dots & 0 & 0 & (-1)^{m-2}b_{2,n} & \dots & (-1)^{m-2}b_{2,1} \\
\vdots & & \ddots & & \vdots & \vdots & \ddots &\vdots \\
0 & 0 & \dots & 1 & 0 & -b_{m-1,n} & \dots & -b_{m-1,1}\\
0 & 0 &\dots  & 0 & 1 & b_{m,n} & \dots & b_{m,1}
\end{pmatrix}
\end{gathered}
\end{equation*}

This reverses the order of the columns in a matrix $B = (b_{ij}) \in M_{m,n}$ and multiplies alternate rows by $-1$.  Under this identification, the minor $[S|T]$ evaluated on $B$ is equal to the maximal minor $\Delta_{(\interval{1}{m} \backslash S) \cup (w_n^o(T)+m)}$ evaluated on $\theta(B)$, so $\theta$ restricts to an embedding $\theta: M_{m,n}^{tnn} \hookrightarrow Gr(m,m+n)^{tnn}$ (see e.g. \cite[\S3]{Postnikov} for more details).  We will often identify $M_{m,n}^{tnn}$ with its image under $\theta$, where the identification between index sets for minors and Pl\"ucker coordinates is given by
\begin{equation}\label{eq:translate minors and plucker coordinates}
\begin{gathered}
\ [S|T]  \longleftrightarrow (\interval{1}{m} \backslash S) \cup (w_n^o(T) + m),\\
I  \longleftrightarrow \Big[\interval{1}{m} \backslash p_1(I) \big| w_n^o(p_2(I) - m)\Big],
\end{gathered}
\end{equation}
where $S \subseteq \interval{1}{m}$, $T \subseteq \interval{1}{n}$, and $I \subseteq \interval{1}{m+n}$.  Recall that the projections $p_1$, $p_2$ were defined in \S\ref{ss:background notation}.

\begin{example}\label{ex:identify minors and plucker coords for 2x2 case}We apply $\theta$ to $B = \left(\begin{smallmatrix}a& b \\ c& d \end{smallmatrix}\right) \in M_{2,2}$:
\[\begin{pmatrix} a& b \\ c & d \end{pmatrix} \mapsto \begin{pmatrix}1 & 0 & -b & -a \\ 0 & 1 & d & c \end{pmatrix}.\]
The matrix $B$ has five minors: $[1|1], [1|2], [2|1], [2|2]$, and the $2\times 2$ determinant $[12|12]$.  These agree with the maximal minors $\Delta_{24}, \Delta_{23}, \Delta_{14}, \Delta_{13}, \Delta_{34}$ of $\theta(B)$ respectively.  The one remaining maximal minor, $\Delta_{12}$, will always be equal to 1 in this setting.
\end{example}

\begin{definition}\label{def:postroid cell}Let $Z \subset \binom{[d]}{k}$ be any family of $k$-subsets.  The \textit{totally nonnegative cell} associated to $Z$ in $Gr(k,d)^{tnn}$ is the set
\[S_Z = \{M \in Gr(k,d)^{tnn} : \Delta_I(M) = 0 \text{ if and only if } I \in Z\}.\]
\end{definition}
Of course, for an arbitrary choice of $Z$ the cell $S_Z$ will often be empty.  In \cite[\S6]{Postnikov}, Postnikov showed that the non-empty cells in $Gr(k,d)^{tnn}$ are parametrized by the set of all Cauchon diagrams on Young diagrams with at most $k$ rows and $(d-k)$ columns.  

Via the embedding $\theta: M_{m,n}^{tnn} \hookrightarrow Gr(m,m+n)^{tnn}$ above, Definition~\ref{def:postroid cell} also induces a definition of the totally nonnegative cells in $M_{m,n}^{tnn}$.  These correspond to the cells $S_Z$ with $\interval{1}{m} \notin Z$, and the non-empty cells are parametrised by the $m\times n$ \textit{rectangular} Cauchon diagrams.

Recall that the $m\times n$ rectangular Cauchon diagrams also parametrise the $\HH$-primes in $\OO_q(M_{m,n})$; this is not a coincidence, as the following theorem of Goodearl, Launois, and Lenagan shows.
\begin{theorem}\label{res:minors in H primes and TNN cells agree}
Let $Z \subseteq \binom{[m+n]}{m}$ be a collection of $m$-subsets with $\interval{1}{m} \notin Z$, and let $Z_q$ be the corresponding family of quantum minors via the identification \eqref{eq:translate minors and plucker coordinates}.  Then the following are equivalent:
\begin{enumerate}
\item The totally nonnegative cell $S_Z$ in $Gr(m,m+n)^{tnn}$ is non-empty.
\item $Z_q$ is a list of all quantum minors belonging to some $\HH$-prime $I_Z$ in $\OO_q(M_{m,n})$.
\end{enumerate}
This defines a bijection between $\HH$-$spec(\OO_q(M_{m,n}))$ and the set of non-empty TNN cells in $M_{m,n}^{tnn}$.

Further, the Cauchon diagram associated to the $\HH$-prime $I_Z$ is the same as the Cauchon diagram associated to the totally nonnegative cell $S_Z$.
\end{theorem}
\begin{proof}
Combine \cite[Corollary~5.5]{GLL1} and \cite[Theorem~4.1]{GLL2}.
\end{proof}
It is conjectured (but not yet known) that the results of Theorem~\ref{res:minors in H primes and TNN cells agree} should hold for all $\HH$-primes in the quantum Grassmannian $\OO_q(Gr(k,d))$.  We therefore restrict our attention to $\HH$-$spec(\OO_q(M_{m,n}))$ and rectangular Cauchon diagrams in this paper, as we will make use of Theorem~\ref{res:minors in H primes and TNN cells agree} to translate the ring-theoretic problem outlined in the introduction into the setting of total nonnegativity. 

\section{Denominator sets via Grassmann necklaces}\label{s:grassmann necklaces}

We fix integers $m,n \geq 2$ and set $\HH := (\KK^{\times})^{m+n}$.  Let $\HH$ act on $\OO_q(M_{m,n})$ as described in \eqref{eq:H action def}.

Let $J \subsetneq K$ be $\HH$-primes in $\OO_q(M_{m,n})$, and write $A_J$ for the localization of $\OO_q(M_{m,n})/J$ at all of its non-zero $\HH$-eigenvectors. (This is equivalent to taking the graded division ring of the graded-prime ring $\OO_q(M_{m,n})/J$, so this localization always exists.)  Following \cite[Definition 3.8]{BGconjecture}, define
\begin{equation}\label{eq:first def of ZJK}\begin{gathered}
\mathcal{E}_{JK} = \{\text{nonzero $\HH$-eigenvectors }c \in \OO_q(M_{m,n})/J : c \text{ is regular modulo } K/J\},\\
Z_{JK} = \{z \in Z(A_J) : zc \in \OO_q(M_{m,n})/J \text{ for some } c \in \mathcal{E}_{JK}\}.
\end{gathered}\end{equation}
It can be verified directly that $Z_{JK}$ is a well-defined $k$-algebra, and a subalgebra of $Z(A_J)$ \cite[Definition 3.8]{BGconjecture}.

If $\mathcal{E}_{JK}$ forms a denominator set in $\OO_q(M_{m,n})/J$, then $Z_{JK}$ nothing but the algebra $Z(\OO_q(M_{m,n})/J[\mathcal{E}_{JK}^{-1}])$.  However, unless $J = K$ this need not be true in general, so we would like to replace $\mathcal{E}_{JK}$ with a smaller multiplicative set $E_{JK} \subset \mathcal{E}_{JK}$ at which we \textit{can} localize.  The following lemma tells us under what conditions this is possible.

\begin{lemma}\label{res:brown goodearl ore set lemma} \cite[Lemma 3.9]{BGconjecture}
Let $J \subsetneq K$ be $\HH$-primes in $\OO_q(M_{m,n})$.  If there exists a denominator set $E_{JK} \subseteq \mathcal{E}_{JK}$ such that
\begin{equation}\label{eq:condition for EJK to satisfy}E_{JK} \cap (L/J) \neq \emptyset \text{ for all $\HH$-primes $L \supseteq J$ such that }L \not\subseteq K,\end{equation}
then there is an equality
\[Z_{JK} = Z\big(\OO_q(M_{m,n})/J[E_{JK}^{-1}]\big).\]
\end{lemma}

In this section we will explicitly construct denominator sets $E_{JK}$ for all pairs of $\HH$-primes $J \subsetneq K$ in $\OO_q(M_{m,n})$.

\subsection{A restatement of the problem}\label{ss: restatement of problem}

There are several elementary observations we can make to simplify the conditions in Lemma~\ref{res:brown goodearl ore set lemma}.  First, if we construct a set $E_K:=E_{0K}$ satisfying \eqref{eq:condition for EJK to satisfy} above for $J = \{0\}$, then $E_K \cap J = \emptyset$ for any $J \subseteq K$ and we can take $E_{JK}$ to be the image of $E_K$ in $\OO_q(M_{m,n})/J$.  Further, since $q$ is not a root of unity, all $\HH$-primes are completely prime by \cite[Theorem II.5.14]{BGbook} and so $c \in \OO_q(M_{m,n})$ is regular modulo $K$ if and only if $c \notin K$.  

By \cite[Proposition~10.7]{GW1} all Ore sets are denominator sets in this setting, so it suffices to check that our sets $E_{K}$ satisfy the Ore conditions.  In fact, we can do even better than this: by \cite{MinorsOreSets}, any quantum minor generates an Ore set in $\OO_q(M_{m,n})$, so it is enough to find a generating set for $E_K$ that consists of quantum minors.

Finally, we also impose the condition that $E_K$ should be finitely generated, in order to simplify the study of the algebras $Z_{JK}$ in future.

We therefore rephrase the problem as follows: for each $K \in \HH$-$spec(\OO_q(M_{m,n}))$, we would like to find a multiplicatively closed set $E_K \subset \OO_q(M_{m,n})$ which is generated by finitely many quantum minors and satisfies:
\begin{itemize}
\item $E_K \cap K = \emptyset$.
\item If $L \in \HH$-$spec(\OO_q(M_{m,n}))$ and $L \not\subseteq K$, then $L \cap E_K \neq \emptyset$.
\end{itemize}
Following the notation of \cite{FryerYakimov}, we will call a set with these properties a \textit{separating Ore set} for the $\HH$-prime $K$.

\subsection{Translation to total nonnegativity}

Constructing separating Ore sets in $\OO_q(M_{m,n})$ turns out to be a difficult ring-theoretic question, requiring us to be able to identify exactly which $\HH$-primes a given quantum minor belongs to.  On the combinatorical side, however, the question of identifying whether a minor is zero on a given cell has been completely solved: see Theorem~\ref{res:oh's result on minors in a cell via GN} below.  The statement of this result requires a combinatorial object called the Grassmann necklace, which we now define.

\begin{definition}\label{def:grassmann necklace}
Let $k,d$ be positive integers with $k < d$, and $\mathcal{I} = (I_1, \dots, I_d)$ a sequence of $k$-subsets of $\interval{1}{d}$.  Then $\mathcal{I}$ is a \textit{Grassmann necklace of type $(k,d)$} if it satisfies:
\begin{itemize}
\item If $i \in I_i$ then $I_{i+1} = (I_i \backslash \{i\}) \cup \{j\}$ for some $j \in \interval{1}{d}$.
\item If $i \not\in I_i$ then $I_{i+1} = I_i$.
\end{itemize}
for each $i \in \interval{1}{d}$.  Note that all indices are taken modulo $d$.
\end{definition}
\begin{definition}\label{def:gale orders}
For each $i \in \interval{1}{d}$, define a total order $\leq_i$ on $\interval{1}{d}$ by
\[i <_i i+1 <_i \dots <_i d <_i 1 <_i \dots <_i i-1.\]
This induces a partial order (also denoted $\leq_i$) on $k$-subsets of $\interval{1}{d}$ as follows: if $S = \{s_1 <_i s_2 <_i \dots <_i s_k\}$ and $T = \{t_1 <_i t_2 <_i \dots <_i t_k\}$ are $k$-subsets of $\interval{1}{d}$, then
\[S \leq_i T \text{ if and only if } s_r \leq_i t_r \text{ for all }r,\ 1 \leq r \leq k.\]
An undecorated symbol $\leq$ will always mean the standard order, i.e. $\leq_1$.
\end{definition}

In \cite[Theorem~17.1]{Postnikov}, Postnikov showed that the non-empty TNN cells in $Gr(k,d)^{tnn}$ are in bijection with the Grassmann necklaces of type $(k,d)$.  The following theorem of Suho Oh uses the Grassmann necklace of a TNN cell to characterise exactly which minors are zero on elements of that cell.

\begin{theorem}\label{res:oh's result on minors in a cell via GN}\cite[Theorem~8]{Oh}
Let $\II = (I_1, \dots, I_d)$ be a Grassmann necklace of type $(k,d)$, and $S_{\II}$ the associated TNN cell in $Gr(k,d)^{tnn}$.  Then the full list of the minors which are zero on elements of $S_{\II}$ is given by
\[\left\{\Delta_T : \exists i \in \interval{1}{d} \text{ such that  } I_i \not\leq_i T \right\}.\]
\end{theorem}
\begin{remark}\label{rem:grassmann necklace elements are nonzero minors}
For any Grassmann necklace $\II = (I_1, \dots, I_n)$, we always have $I_i \leq_i I_j$ for all $i,j$ by \cite[Lemma 4.4]{OPS}.  Combining this with Theorem~\ref{res:oh's result on minors in a cell via GN}, we find that $\Delta_{I_i}(A) > 0$ for all $A \in S_{\II}$ and all $i \in \interval{1}{d}$.  In other words, all terms in the Grassmann necklace of a cell define minors which are strictly positive on that cell.
\end{remark}
\begin{remark}\label{rem:GN of matrix cells}
The Grassmann necklaces of TNN cells in $M_{m,n}^{tnn}$ are exactly those with $I_1 = \interval{1}{m}$.  
\end{remark}

In \S\ref{ss:chain rooted at,etc} below, we will describe how to read off the Grassmann necklace of a cell from its Cauchon diagram; we therefore postpone any examples of Grassmann necklaces until then.

\subsection{Constructing Ore sets $E_K$ in terms of Grassmann necklaces}

The Grassmann necklace turns out to be exactly what we need to construct our separating Ore sets.

Let $K$ be a $\HH$-prime in $\OO_q(M_{m,n})$.  Applying the identifications outlined in \S\ref{ss:background_tnn}, we will use ``the Grassmann necklace of $K$'' as shorthand for ``the Grassmann necklace associated to the TNN cell in $Gr(m,m+n)^{tnn}$ with the same Cauchon diagram as $K$''.  Recall that the identification between quantum minors and Pl\"ucker coordinates is given in \eqref{eq:translate minors and plucker coordinates}.

We are now in a position to state our first main theorem.  Recall that the projections $p_1$, $p_2$ were defined in \S\ref{ss:background notation}.

\begin{theorem}\label{res:GN constructs the Ore set we need}
Let $K \in \HH$-$spec(\OO_q(M_{m,n}))$ and $q \in \KK^{\times}$ not a root of unity, and write $\II_K = (K_1, \dots, K_{m+n})$ for the Grassmann necklace of $K$.  Define the following set of quantum minors:
\[\widetilde{E}_K = \Big\{[S_i|T_i]_q : S_i = \interval{1}{m} \backslash p_1(K_i),\ T_i = w_n^o(p_2(K_i) - m),\ 2 \leq i \leq m+n\Big\}\]
Then $\widetilde{E}_K \cap K = \emptyset$, and if $L$ is any other $\HH$-prime with $L \not\subseteq K$, then $\widetilde{E}_K \cap L \neq \emptyset$.
\end{theorem}
\begin{proof}
Let $C_K$ be the Cauchon diagram of the $\HH$-prime $K$, and write $S_K$ for the totally nonnegative cell associated to $C_K$.  By Theorem~\ref{res:minors in H primes and TNN cells agree}, a quantum minor belongs to $K$ if and only if the corresponding Pl\"ucker coordinate is zero on elements of the cell $S_K$.  Since the quantum minors in $\widetilde{E}_K$ are exactly those corresponding to the Grassmann necklace of $K$, it follows immediately from Theorem~\ref{res:oh's result on minors in a cell via GN} and Remark~\ref{rem:grassmann necklace elements are nonzero minors} that $\widetilde{E}_K \cap K = \emptyset$.

Now suppose $L \in \HH$-$spec(\OO_q(M_{m,n}))$ is such that $L \not\subseteq K$, and let $\II_L := (L_1, \dots, L_{m+n})$ be the corresponding Grassmann necklace.  We claim that there is some $i \in \interval{2}{m+n}$ such that $K_i \not\geq_i L_i$.  (We can ignore the case $i=1$ since $K_1 = L_1 = \interval{1}{m}$.)

Indeed, since every $\HH$-prime in $\OO_q(M_{m,n})$ is generated by quantum minors \cite[Theorem~4.4.1]{Casteels2}, there is some quantum minor $[S|T]_q \in L \backslash K$.  Define $\Lambda := (\interval{1}{m}\backslash S) \cup (w_m^o(T)+m)$, so that $[S|T]_q \not\in K$ is equivalent to $\Delta_{\Lambda}$ being non-zero on the cell $S_K$ (Theorem~\ref{res:minors in H primes and TNN cells agree}).  By Theorem~\ref{res:oh's result on minors in a cell via GN}, we have $\Lambda \geq_i K_i$ for all $i$.

If $K_i \geq_i L_i$ for all $i$, then we also have $\Lambda \geq_i L_i$ for all $i$, and hence $[S|T]_q \not\in L$; this is a contradiction.  

Therefore there must be at least one $i$ such that $K_i \not\geq_i L_i$.  Applying Theorem~\ref{res:oh's result on minors in a cell via GN} again, we see that $\Delta_{K_i}$ must be zero on the cell $S_L$ and hence the corresponding quantum minor $[S_i|T_i]_q$ belongs to $L$.  We have shown that $L \cap \widetilde{E}_K \neq \emptyset$, as required.
\end{proof}

\begin{theorem}\label{res:we have constructed ore sets yay}
Let $q$, $K$, $\widetilde{E}_K$ be as in Theorem~\ref{res:GN constructs the Ore set we need}, and let $J$ be any other $\HH$-prime in $\OO_q(M_{m,n})$ with $J \subsetneq K$.  Let $E_K$ be the multiplicative set generated by $\widetilde{E}_K$ in $\OO_q(M_{m,n})$, and $E_{JK}$ the image of $E_K$ in $\OO_q(M_{m,n})/J$.  Then we have an equality
\[Z_{JK} = Z\Big(\big(\OO_q(M_{m,n})/J\big)[E_{JK}^{-1}]\Big).\]
\end{theorem}
\begin{proof}
By \cite{MinorsOreSets} and \cite[Proposition~10.7]{GW1}, $E_K$ is a denominator set in $\OO_q(M_{m,n})$.  Complete primality implies that we have $E_K \cap K = \emptyset$, and hence $E_{JK}$ is well behaved for all $J \subsetneq K$.  Finally, $E_{JK}$ must be a denominator set in $\OO_q(M_{m,n})/J$ by the universality of localization and quotients.  The conclusion of the theorem now follows directly from Theorem~\ref{res:GN constructs the Ore set we need} and Lemma~\ref{res:brown goodearl ore set lemma}.
\end{proof}

\subsection{Computing the Grassmann necklace of a cell}\label{ss:chain rooted at,etc}
While Theorem~\ref{res:we have constructed ore sets yay} guarantees that the desired Ore sets will always exist, it is very quiet on the subject of how to actually construct such a collection for a given $\HH$-prime $K$.

In this section we introduce Oh's ``chain rooted at a square'' construction from \cite{Oh}, which can be used to read the Grassmann necklace directly off the corresponding Cauchon diagram.  Since we will make repeated use of these chains in later sections, we first give a detailed proof of their properties.

Let $C$ be an $m\times n$ Cauchon diagram, with rows numbered with $\interval{1}{m}$ from top to bottom, and columns numbered with $\interval{1}{n}$ from left to right.  We will often want to refer to regions northwest of a square, which we do as follows: 
\begin{definition}\label{def:region northwest}
If $(x,y)$ is a square in $C$, the region \textbf{weakly northwest} of $(x,y)$ consists of the squares $\{(a,b):a \leq x, b \leq y\}$, and the region \textbf{strictly northwest} of $(x,y)$ consists of the squares $\{(a,b):a < x,b < y\}$.
\end{definition}

The next lemma shows that we can always use the Cauchon condition to find a unique ``nearest'' white square northwest of a given square.  

\begin{lemma}\label{res:nearest white square in chain} Let $C$ be a Cauchon diagram, and $(x,y)$ a square in $C$.
\begin{enumerate}
\item If $(x,y)$ is a white square and the region strictly north-west of $(x,y)$ is not all black, there is a unique white square $(a,b)$ in this region such that $|x-a|$ and $|y-b|$ are simultaneously minimised.
\item If $(x,y)$ is a black square and the region weakly north-west of $(x,y)$ is not all black, there is a unique white square $(a,b)$ in this region such that $|x-a|$ and $|y-b|$ are simultaneously minimised.
\end{enumerate}
\end{lemma}
\begin{proof}
First note that deleting rows from the bottom of $C$ or columns from the right of $C$ has no effect on whether $C$ is a Cauchon diagram or the existence of the promised square $(a,b)$, so it is enough to prove the lemma for the bottom right square of $C$. This allows us to dispense with repetition of phrases like ``in the region northwest of $(x,y)$''.  

So let $(x,y)$ be the bottom right square of $C$, and suppose that it is a white square.  Let $(a,b)$ be the white square chosen by first minimising $|x-a|$ and then $|y-b|$ (subject to the conditions $a \neq x$ and $b \neq y$) and let $(a',b')$ be the white square obtained by minimising $|y-b'|$ first and then $|x-a'|$ afterwards (again subject to $a' \neq x$ and $a' \neq y$).  If $(a,b) \neq (a',b')$ then we have $a \geq a'$ and $b' \geq b$, and at least one of these inequalities is strict.  This forces the square $(a,b')$ to also be white by the Cauchon property of $C$ (see Figure~\ref{fig:unique chain}).  However, this contradicts our choice of either $a$ or $b'$ (both of which were chosen without constraint), so we must have $(a,b) = (a',b')$, and this is the unique white square promised by the lemma.

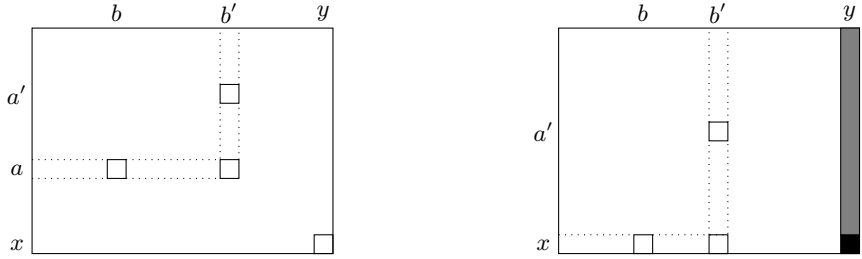
\begin{figure}[h!]

\begin{tikzpicture}
\draw rectangle (4,3);

\draw[dotted] (0,1) -- ++(1,0) ++(0.25,0) -- ++(1.25,0) ++(0.25,0.25) -- ++(0,0.75) ++ (0,0.25) -- ++(0,0.75);
\draw[dotted] (0,1.25) -- ++(1,0) ++ (0.25,0) -- ++( 1.25,0)  -- ++(0,0.75) ++(0,0.25) -- ++(0,0.75);

\draw (1,1) rectangle +(0.25,0.25)  (2.5,2) rectangle +(0.25,0.25) (3.75,0) rectangle +(0.25,0.25);
\draw (2.5,1) rectangle +(0.25,0.25);

\draw (-0.2,0.12) node{\footnotesize $x$} ++(0,1) node{\footnotesize $a$} ++(0,1) node{\footnotesize $a'$};
\draw (1.12,3.2) node{\footnotesize $b$} ++(1.5,0) node{\footnotesize $b'$} ++(1.25,0) node{\footnotesize $y$};

\draw (7,0) rectangle (11,3);

\draw[dotted] (7,0.25) -- ++(1,0) ++(0.25,0) -- ++(1,0) -- ++(0,1.25) ++(0,0.25) -- ++(0,1.25);
\draw[dotted] (9,0.25) -- ++(0,1.25) ++(0,0.25) -- ++(0,1.25);

\draw (8,0) rectangle +(0.25,0.25)  (9,1.5) rectangle +(0.25,0.25);
\draw (9,0) rectangle +(0.25,0.25);
\draw[fill=black] (10.75,0) rectangle +(0.25,0.25);
\draw[fill=gray] (10.75,0.25) rectangle +(0.25,2.75);

\draw (6.8,0.12) node{\footnotesize $x$} ++(0,1.5) node{\footnotesize $a'$};
\draw (8.12,3.2) node{\footnotesize $b$} ++(1,0) node{\footnotesize $b'$} ++(1.75,0) node{\footnotesize $y$};
\end{tikzpicture}

\caption{The Cauchon condition on the diagram guarantees a unique nearest white square.}
\label{fig:unique chain}
\end{figure}

Now suppose $(x,y)$ is a black square, so at least one of row $x$ or column $y$ is all black in $C$.  If both are all black then any white squares will be strictly northwest of $(x,y)$ and the previous argument applies.  Otherwise, suppose there is at least one white square in row $x$ (the argument for column $y$ is symmetric), and let $(x,b)$ be the white square in this row such that $|y-b|$ is minimised.  If there is another white square $(a',b')$ anywhere in the diagram with $|y-b'|<|y-b|$, then the Cauchon property implies that $(x,b')$ is also white: a contradiction to our choice of $b$.
\end{proof}

We will refer to the square $(a,b)$ constructed in Lemma~\ref{res:nearest white square in chain} as the \textbf{white square nearest to $(x,y)$}.  It is clear from the construction that the step from $(x,y)$ to its nearest white square will always have the pattern of black squares depicted in Figure~\ref{fig:special case of lacunary condition}.  This is similar to the lacunary sequence construction used in \cite{LL1}; see Remark~\ref{rem: lacunary sequences vs chain rooted at} below.

\begin{figure}[h!]

\begin{tikzpicture}

\draw[gray,fill=gray] (0,1) rectangle (3,1.75) (1.5,1.75) rectangle (3,3);

\draw rectangle (4,3);
\draw (3,0.75) rectangle +(0.25,0.25) (1.25,1.75) rectangle +(0.25,0.25);

\draw[-] (0,1) -- (3,1) -- (3,3) (0,1.75)--(1.25,1.75) (1.5,2)--(1.5,3);

\draw (-0.2,0.87) node{\footnotesize $x$} ++(0,1) node{\footnotesize $a$};
\draw (1.37,3.2) node{\footnotesize $b$} ++(1.75,0) node{\footnotesize $y$};

\end{tikzpicture}

\caption{The step between any square and its nearest white square guarantees that all squares in the shaded region will be black.}
\label{fig:special case of lacunary condition}
\end{figure}
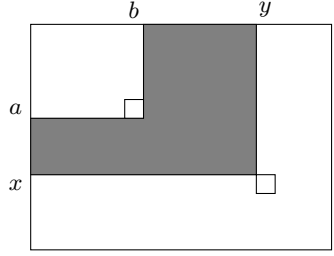

The following definition is a restatement of the definition originally introduced by Oh in \cite{Oh}.

\begin{definition}\label{def:chain rooted at}
Let $(x,y)$ be a square in a Cauchon diagram $C$ such that there is at least one white square weakly northwest of $(x,y)$.  The \textbf{chain rooted at $(x,y)$} is the sequence $(x_1,y_1),\dots,(x_t,y_t)$ constructed as follows:
\begin{itemize}
\item Initial step:
\begin{itemize}
\item If $(x,y)$ is a black square, set $(x_1,y_1)$ to be the unique nearest white square weakly northwest of $(x,y)$.
\item If $(x,y)$ is a white square, $(x_1,y_1) := (x,y)$.
\end{itemize}
\item For $i \geq 1$: 
\begin{itemize}
\item If there are no white squares strictly northwest of $(x_i,y_i)$, then $t:=i$ and the sequence terminates.
\item Otherwise, take $(x_{i+1},y_{i+1})$ to be the nearest white square strictly northwest of $(x_i,y_i)$.
\end{itemize}
\end{itemize}
Finally, if $(x,y)$ is a black square and there is no white square weakly northwest of it, define the chain rooted at $(x,y)$ to be the empty chain.
\end{definition}

The chain $(x_1,y_1),\dots, (x_t,y_t)$ rooted at square $(x,y)$ naturally corresponds to the $t\times t$ quantum minor $[x_t, \dots, x_1|y_t,\dots, y_1]_q$.  Recall from \eqref{eq:translate minors and plucker coordinates} that the corresponding Pl\"ucker coordinate has columns indexed by the $m$-subset
\[\Big(\interval{1}{m}\backslash\{x_1, \dots, x_t\}\Big) \cup \big\{w_{n}^o(y_i) + m: 1 \leq i \leq t\big\}.\]

An important use of the chain construction in Definition~\ref{def:chain rooted at} is given in the following theorem.

\begin{proposition}\label{res:oh grassmann necklace computation} \cite[Proposition 14]{Oh}
Let $C$ be an $m \times n$ rectangular Cauchon diagram, and $\II = (I_1, \dots, I_{m+n})$ the associated Grassmann necklace.  Label the boxes along the bottom row of $C$ with the numbers $\interval{1}{n}$ from left to right, and the boxes in the rightmost column with $\interval{n}{m+n-1}$ from bottom to top.  Then 
\[I_{m+n-k+1} = \Big(\interval{1}{m} \backslash \{x_1, \dots, x_t\}\Big) \cup \{w_n^o(y_i) + m: 1 \leq i \leq t\},\] 
where $(x_1,y_1), \dots, (x_t,y_t)$ is the chain rooted at box $k$ for $k \in \interval{1}{m+n-1}$, and $I_1:= \interval{1}{m}$.
\end{proposition}

Note that the numbering of the boxes along the southeast border ran in the opposite direction in the version of Proposition~\ref{res:oh grassmann necklace computation} stated in \cite{Oh}; this is why the chain rooted at box $k$ corresponds to the Grassmann necklace term $I_{m+n-k+1}$ in our setting.  The reason for this change of direction is to simplify the notation in \S\ref{s:permutation version} below, where we will associate the chain rooted at box $k$ to the image of the interval $\interval{1}{k}$ under a certain permutation.

\begin{example}\label{ex:computing a GN via chains}
We compute the Grassmann necklace of the following Cauchon diagram:
\begin{equation*}
\begin{tikzpicture}[scale=0.75]
\draw[step=1] grid (4,3);
\draw[fill=black] (0,0) rectangle +(1,1) (0,1) rectangle +(1,1) (1,1) rectangle +(1,1) (2,1) rectangle +(1,1) (2,2) rectangle +(1,1);
\node[white] at (0.5,0.5) {$1$};
\node at (1.5,0.5) {$2$};
\node at (2.5,0.5) {$3$};
\node at (3.5,0.5) {$4$};
\node at (3.5,1.5) {$5$};
\node at (3.5,2.5) {$6$};

\node[gray] at (-0.3,2.5) {\footnotesize $1$};
\node[gray] at (-0.3,1.5) {\footnotesize $2$};
\node[gray] at (-0.3,0.5) {\footnotesize $3$};

\node[gray] at (0.5,3.3) {\footnotesize $1$};
\node[gray] at (1.5,3.3) {\footnotesize $2$};
\node[gray] at (2.5,3.3) {\footnotesize $3$};
\node[gray] at (3.5,3.3) {\footnotesize $4$};

\end{tikzpicture}
\end{equation*}
The chains rooted at boxes $1$ through $6$ are:
\begin{align*}
1&: (1,1)\\
2&: (3,2), (1,1)\\
3&: (3,3), (1,2)\\
4&: (3,4), (1,2)\\
5&: (2,4), (1,2)\\
6&: (1,4)
\end{align*}
Converting these to the Pl\"ucker coordinate notation as in Proposition~\ref{res:oh grassmann necklace computation}, we get the Grassmann necklace
\[\II = (123,234,346,246,256,267,238).\]
\end{example}

\begin{remark}\label{rem: lacunary sequences vs chain rooted at}
We can think of the chain rooted at $(x,y)$ as a kind of ``reverse lacunary sequence'': compare Definition~\ref{def:chain rooted at} to \cite[Definition~3.1]{LL1}.  The main differences between the two are:
\begin{itemize}
\item A lacunary sequence constructed from a black square defines a minor which belongs to the $\HH$-prime corresponding to that diagram; this is not true for a chain rooted at a black square, since the chain starts from the nearest white square instead.
\item Lacunary sequences from white squares and chains rooted at white squares always define minors which are not in the $\HH$-prime (this is easily seen using the first author's results in \cite[Theorem~5.6]{Casteels1}, since there is an obvious vertex-disjoint path system in each case.)
\item Lacunary sequences from a given square need not be unique, while the chain rooted at a square is always unique.
\end{itemize}
\end{remark}
\begin{example}\label{ex:lacunary sequences fail in sl4}
The following example illustrates why the language of lacunary sequences on its own may not be sufficient to verify that a set $E_K$ has the desired properties.  Minors defined by lacunary sequences are used in \cite{LL1} to give efficient positivity tests for the corresponding TNN cells, but do not give a characterisation of all minors within a cell.  The interested reader is referred to \cite[Definition~3.1]{LL1} for the definition of lacunary sequences.

Let $K$ and $L$ be the $\HH$-primes in $\OO_q(M_{4,4})$ corresponding to the following Cauchon diagrams:

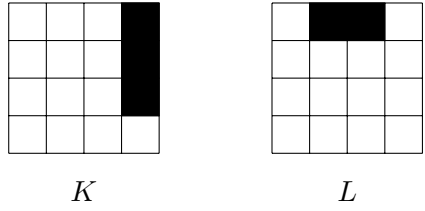
\begin{center}
\begin{tikzpicture}[scale=0.5]
\draw[step=1] grid (4,4);
\draw[fill=black] (3,1) rectangle ++(1,3);
\draw (2,-1) node {$K$};

\draw[step=1] (7,0) grid ++(4,4);
\draw[fill=black] (8,3) rectangle ++(2,1);
\draw (9,-1) node{$L$};
\end{tikzpicture}
\end{center}
Using the formulas in \cite[Definition~2.6]{GLL1} to compute a full list of quantum minors in each $\HH$-prime, we find that the only quantum minors in $L \backslash K$ are $[12|23]_q$ and $[124|234]_q$.  However, neither of these minors are defined by lacunary sequences in $L$.  As a result, no matter how we define the set $E_K$, the lacunary sequences will not be able to ``see'' that $E_K \cap L \neq \emptyset$.

We further note that if we construct $E_K$ as in Theorem~\ref{res:GN constructs the Ore set we need}, we have $E_K \cap L = \{[12|23]_q\}$ (corresponding to the chain rooted at box (2,4) of the Cauchon diagram of $K$, i.e. Grassmann necklace term $I_3$), and in fact $[12|23]_q$ is not a lacunary sequence in $K$ either.
\end{example}

\section{Denominator sets via restricted permutations}\label{s:permutation version}

Thus far we have worked mostly in the settings of ring theory and combinatorics.  There is a third perspective on this question, however, which is the study of $\HH$-primes via representation theory and the language of quantum groups (e.g. Brown, Hodges-Levasseur, Joseph, Yakimov; see bibliography of \cite{Yakimov2}).

In \cite{FryerYakimov} the second author and Yakimov studied the problem of realising the $Z_{JK}$ from \eqref{eq:first def of ZJK} as centres of localizations for a more general class of algebras.  Using the language of quantum groups and Demazure modules, they constructed separating Ore sets for all quantum function algebras $\OO_q(G)$ on complex simple groups $G$, and all quantum Schubert cell algebras $\mathcal{U}^{-}[w]$ coming from symmetrizable Kac-Moody algebras.

The results of \cite{FryerYakimov} do not lend themselves well to explicit computation, however.  In this section we show that in the special case of $\OO_q(M_{m,n})$, the denominator sets constructed in \cite[Main Theorem]{FryerYakimov} and in Theorem~\ref{res:we have constructed ore sets yay} agree with each other, thus providing a new perspective on the results of \cite{FryerYakimov}.

We first need to introduce some new notation.  Recall that $q \in \KK^{\times}$ is not and will never be a root of unity.

\subsection{$\HH$-primes via restricted permutations}\label{ss: background and notation for restricted permutations}

Fix integers $m,n \geq 2$.  If $I \subset\interval{1}{m+n}$, recall that the projections $p_1$ and $p_2$ of $I$ are defined to be
\[p_1(I) = I \cap \interval{1}{m}, \qquad p_2(I) = I \cap \interval{m+1}{m+n}.\]

There is a bijection between $\HH$-$spec(\OO_q(M_{m,n}))$ and the subset of $S_{m+n}$ defined by
\begin{equation}\label{eq:restricted perms bruhat order}\mathcal{S} = \{\sigma \in S_{m+n} : \sigma \preccurlyeq c^m\},\end{equation}
where $\preccurlyeq$ denotes the strong Bruhat order, and $c$ is the Coxeter element in $S_{m+n}$.  These are also known as \textit{restricted permutations}, since another characterisation of this set is
\[\mathcal{S} = \{\sigma \in S_{m+n}: -n \leq \sigma(i) - i \leq m,\ \forall i \in \interval{1}{m+n}\}.\]
(See \cite[Proposition~1.3]{Launois}, \cite[\S1.3]{GLL1}.)

\begin{remark}
There are many different conventions for associating a permutation (or more generally, a pair of Weyl group elements) to a $\HH$-prime or Cauchon diagram.  We follow the convention in \cite{GLL1,Yakimov1} and use the set \eqref{eq:restricted perms bruhat order} above; this associates to each Cauchon diagram the \textit{inverse} of the permutation used in \cite{Launois,MR1,Postnikov}.  Neither of these should be confused with the \textit{decorated permutation} of \cite{Oh,Postnikov}, which is a related but separate construction.
\end{remark}

In \cite[\S19]{Postnikov}, Postnikov showed that we can obtain the restricted permutation of a $\HH$-prime from its Cauchon diagram via pipe dreams as follows: replace each black square in the diagram with a crossing and each white square with a pair of elbows, and arrange the numbers $1, 2, \dots, m+n$ along both the northwest and southeast borders of the diagram as in Figure~\ref{fig:pipedreamillustration}.  We can then read off the permutation by following the pipes from southeast to northwest along the diagram.

\vspace{-7em} 
\begin{figure}[h!]
\begin{tikzpicture}[scale=0.75]

\draw (0,0) grid (4,3) ;
\draw[fill=black] (0,0) rectangle +(1,2) (1,1) rectangle +(2,1) (2,2) rectangle +(1,1);

\draw[->,thick] (4.5,1.5) -- (8,1.5);

\bsq{(5.75,1.8)}
\wsq{(5.75,0.2)}

\draw (9,0) grid (13,3);
\foreach \x in {(9,0),(9,1),(10,1),(11,1),(11,2)}
	{\bsq{\x}; }

\foreach \x in {(9,2),(10,2),(10,0),(11,0),(12,0),(12,1),(12,2)}
	{\wsq{\x};}

\foreach \x in {1,...,4}
{
	\draw (\x,-0.3) +(8.5,0) node{\tiny \x};
}
\foreach \x in {4,...,7}
{
	\draw (\x,3.3) +(5.5,0) node{\tiny \x};
}

\foreach \y in {1,...,3}
{
	\draw (8.7,\y) +(0,-0.5) node{\tiny \y};
}

\foreach \y in {5,...,7}
{
	\draw (13.3,\y) +(0,-4.5) node{\tiny \y};
}

\end{tikzpicture}
\caption{Computing the permutation of a diagram using pipe dreams.}
\label{fig:pipedreamillustration}
\end{figure}
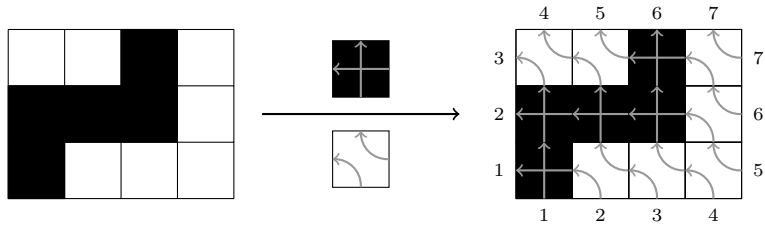

In the example given in Figure~\ref{fig:pipedreamillustration}, we obtain the permutation $\big(\begin{smallmatrix}1 & 2 & 3 & 4 & 5 & 6 & 7 \\ 3 & 1 & 4 & 6 & 2 & 5 & 7 \end{smallmatrix}\big)$.

We can also think of restricted permutations in terms of reduced words in $S_{m+n}$ as follows.  Let $s_i$ denote the elementary transposition $(i\ i+1)$, and associate to each square in the Cauchon diagram an elementary transposition according to the following rule: put $s_1$ in the bottom left corner box, and if a box already contains the transposition $s_i$ then put $s_{i+1}$ in the box directly above it and the box directly to its right (if they exist).  For example, if $C$ was a $3\times 4$ diagram then we would assign the transpositions to squares as illustrated in Figure~\ref{fig:how to assign elementary transpositions}.

\begin{figure}
\begin{tikzpicture}[scale=0.75]
\draw grid (4,3);
\draw (0,0) + (0.5,0.5) node{$s_1$};
\draw (0,1) + (0.5,0.5) node{$s_2$};
\draw (0,2) + (0.5,0.5) node{$s_3$};
\draw (1,0) + (0.5,0.5) node{$s_2$};
\draw (1,1) + (0.5,0.5) node{$s_3$};
\draw (1,2) + (0.5,0.5) node{$s_4$};
\draw (2,0) + (0.5,0.5) node{$s_3$};
\draw (2,1) + (0.5,0.5) node{$s_4$};
\draw (2,2) + (0.5,0.5) node{$s_5$};
\draw (3,0) + (0.5,0.5) node{$s_4$};
\draw (3,1) + (0.5,0.5) node{$s_5$};
\draw (3,2) + (0.5,0.5) node{$s_6$};
\end{tikzpicture}
\caption{Assigning an elementary transposition to each square in a $3\times 4$ diagram.}
\label{fig:how to assign elementary transpositions}
\end{figure}

Let $w$ denote the word obtained by reading off all of these transpositions from left to right, top to bottom.  (This is exactly the permutation $c^m$ in \eqref{eq:restricted perms bruhat order}.)  For example, in the $3 \times 4$ case we have
\[w = (s_3s_4s_5s_6)(s_2s_3s_4s_5)(s_1s_2s_3s_4).\]
The restricted permutation $v$ associated to a given Cauchon diagram $C$ is the subword of $w$ obtained by removing from $w$ all $s_i$ which appear in white squares of $C$.  See \eqref{eq:example of restricted permutation} below for an example.

If $v$ is the permutation associated to $C$ and $(x,y)$ is a square in row $x$ and column $y$ of $C$, we define the permutation $v_{x,y}$ to be the subword of $v$ obtained by colouring white all squares of $C$ which are strictly below row $x$ or strictly to the right of column $y$, and then reading off the resulting permutation from the diagram as above.

To clarify: when reading the transpositions from the Cauchon diagram as described above, we write them down from left to right and compose them from right to left, so the transposition nearest the bottom right of the diagram is always the first to be applied.  For the Cauchon diagram in Figure~\ref{fig:pipedreamillustration}, we therefore obtain
\begin{equation}\label{eq:example of restricted permutation}v = s_5s_2s_3s_4s_1 = \big(\begin{smallmatrix}1 & 2 & 3 & 4 & 5 & 6 & 7 \\ 3 & 1 & 4 & 6 & 2 & 5 & 7 \end{smallmatrix}\big).\end{equation}


In \S\ref{s:grassmann necklaces}, we numbered the rows of a Cauchon diagram with $\interval{1}{m}$ from top to bottom, and the columns with $\interval{1}{n}$ from left to right: this is the standard convention.  In order to make the proofs in this section more readable, we introduce the following new numbering for the rows and columns of a Cauchon diagram $C$ (see also Figure~\ref{fig:pipedreamnumbering}):
\begin{itemize}
\item Number the rows of $C$ with $\interval{1}{m}$ from bottom to top.
\item Number the columns of $C$ with $\interval{m+1}{m+n}$ from left to right.
\end{itemize}
This is because we will be using the image of certain intervals under restricted permutations in order to construct minors, and this convention matches the pipedream numbering in Figure~\ref{fig:pipedreamillustration}.

\begin{figure}[b]
\begin{tikzpicture}
\draw[light-gray,dotted,step=0.5] grid (4,3);
\draw rectangle (4,3);
\draw (-0.3,0.25) node{\footnotesize 1};
\draw (-0.3,0.75) node{\footnotesize 2};
\draw (-0.3,1.8) node{$\vdots$};
\draw (-0.3,2.75) node{\footnotesize $m$};
\node[rotate=-90] at (0.25,3.6) {\footnotesize $m+1$};
\node[rotate=-90] at (0.75,3.6) {\footnotesize $m+2$};
\draw (2.25,3.6) node{$\dots$};
\node[rotate=-90] at (3.75,3.6) {\footnotesize $m+n$};
\end{tikzpicture}

\caption{Numbering the rows and columns of the Cauchon diagram according to pipedream convention.}
\label{fig:pipedreamnumbering}
\end{figure}
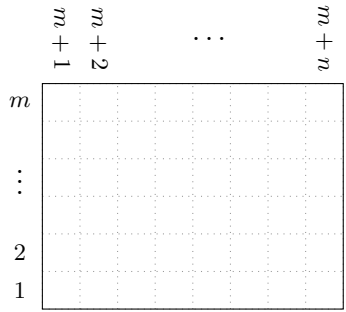

The rows and columns of quantum minors will continue to be expressed in the standard convention; the reader can easily translate between the two by applying $w_m^o$ to the row set and adding or subtracting $m$ from each term in the column set as required (see also Figure~\ref{fig:translate minors pipedreams}).

\subsection{From restricted permutations to chains}\label{ss:res perms to chains}
In this section we will prove Theorem~\ref{res: denominator sets agree, intro version}, i.e. that the Ore sets in \cite[Main Theorem]{FryerYakimov} restricted to the case $\OO_q(M_{m,n})$ agree with those constructed via Grassmann necklaces in Theorem~\ref{res:we have constructed ore sets yay} above.

We first state the relevant part of \cite[Main Theorem]{FryerYakimov}, translated into the language of this paper.
\begin{theorem}\label{res:sian-milen thm}
Let $\KK$ be an infinite field of arbitrary characteristic and $q \in \KK^{\times}$ not a root of unity.  Fix a $\HH$-prime $K$ in $\OO_q(M_{m,n})$, let $v$ be the restricted permutation associated to $K$, and define $I_k:=v\interval{1}{k}$ for $k \in \interval{1}{m+n-1}$.  The the multiplicative set generated by the following minors:
\begin{gather*}
\Big[w_m^o\big(p_1(I_k)\big)\Big|\interval{m+1}{m+k}\backslash p_2(I_k) - m\Big]_q \quad \text{for } k \in \interval{1}{n},\\
\Big[w_m^o\big(p_1(I_k)\backslash \interval{1}{k-n}\big)\Big| \interval{m+1}{m+n} \backslash p_2(I_k) - m\Big]_q \quad \text{for } k \in \interval{n+1}{m+n-1},
\end{gather*}
is a separating Ore set for $K$ in $\OO_q(M_{m,n})$.
\end{theorem}
\begin{proof}
We refer the reader to \cite{FryerYakimov,Yakimov1} for notation and definitions.  The Ore sets in \cite[Main Theorem]{FryerYakimov} are constructed for general quantum Schubert cell algebras $\mathcal{U}^{-}[w]$; we obtain an algebra isomorphic to $\OO_q(M_{m,n})$ by restricting to the special case where $\mathcal{U}^{-}$ is the negative part of $\mathcal{U}_q(\mathfrak{sl}_{m+n})$ and $w = c^m$, i.e. the $m$th power of the Coxeter element $c = (12\dots m+n)$ \cite[Lemma 4.1]{Yakimov1}.  Further, \cite[Lemma~4.3]{Yakimov1} gives an explicit map which translates between the two settings.

Now the quantum minors in the statement of the theorem are exactly the image in $\OO_q(M_{m,n})$ of the elements $\{d_{v,\varpi_k} : 1 \leq k \leq m+n-1\}$ from \cite[Main Theorem]{FryerYakimov}, where $\varpi_k$ denotes the fundamental weight $(\underbrace{1,1,\dots, 1}_{k},0, \dots,0)$.  The fact that these generate the whole Ore set in \cite[Main Theorem]{FryerYakimov} follows directly from equation (2.4) of \cite{FryerYakimov}.
\end{proof}

We will relate the sets $v\interval{1}{k}$ from Theorem~\ref{res:sian-milen thm} to the chains rooted at boxes along the south-east border of the corresponding Cauchon diagram, and hence show that they define the same minors as the Grassmann necklace.

To do this, we first show that we can recover the chain rooted at $(x,y)$ from $v_{x,y}$ and $w_{x,y}$, i.e. the subwords of $v$ and $w$ obtained by ignoring any squares not weakly northwest of $(x,y)$.  This approach is inspired by a similar technique used by Talaska and Williams in \cite{TW1}.

\begin{proposition}\label{res:chains in terms of permutations}
Let $C$ be an $m\times n$ Cauchon diagram, let $(x,y)$ be any square in $C$, and let $(x_1,y_1),(x_2,y_2),\dots, (x_t,y_t)$ be the chain rooted at $(x,y)$ with respect to the numbering convention in Figure~\ref{fig:pipedreamnumbering}.  If $v$ is the permutation associated to $C$, $w$ the permutation associated to the $m\times n$ diagram with all black squares, and $v_{x,y}$, $w_{x,y}$ the corresponding permutations obtained by restricting to the diagram weakly northwest of $(x,y)$, then
\begin{equation}\label{eq:final form of I for chain/permutation result}v_{x,y}w_{x,y}^{-1}\interval{1}{m} = \Big(\interval{1}{m}\backslash \{x_1,\dots,x_t\}\Big) \cup \{y_t,\dots, y_1\}.\end{equation}
\end{proposition}
\begin{proof}
We highlight the following fact, since we will use it repeatedly: the elementary transposition in box $(a,b)$ is $s_{a+b-m-1}$.  Under the row/column convention in Figure~\ref{fig:pipedreamnumbering}, box $(a,b)$ is in the $a$th row up from the bottom of $C$ and the $(b-m)$th column from the left.

Note that if the chain rooted at $(x,y)$ is the empty chain (i.e. all boxes weakly northwest of $(x,y)$ are black) then $v_{x,y} = w_{x,y}$ and the result is certainly true in this case.  So from now on, we will assume that the chain rooted at $(x,y)$ is not empty.

The first step is to compute the action of $w_{x,y}^{-1}$ on $\interval{1}{m}$.  The permutation $w_{x,y}^{-1}$ is given in terms of elementary transpositions by
\begin{equation}\label{eq: permutation w restricted to x,y}w_{x,y}^{-1} = (s_{x+y-m-1}\dots s_x)(s_{x+y-m}\dots s_{x+1})\dots(s_{y-2}\dots s_{m-1})(s_{y-1}\dots s_m).\end{equation}
This can be visualised by colouring all boxes weakly northwest of $(x,y)$ black and taking the inverse pipe-dream, i.e. moving left to right and top to bottom along the pipes.  It follows from \eqref{eq: permutation w restricted to x,y} that $w_{x,y}^{-1}$ fixes $\interval{1}{x-1}$ and acts as $j \mapsto j + y-m$ for $j \in \interval{x}{m}$.  We therefore have
\begin{equation}\label{eq:set for all white diagram}w_{x,y}^{-1}\interval{1}{m} = \interval{1}{x-1} \cup \interval{x+y-m}{y},\end{equation}
which corresponds to the chain rooted at $(x,y)$ on the all-white $m\times n$ diagram.  We now need to show that $v_{x,y}$ ``corrects'' this chain to fit the given diagram $C$.

The action of $v_{x,y}$ can be computed by reading right-to-left, bottom-to-top from square $(x,y)$ in $C$ and applying any transpositions $s_i$ in black squares as we come to them.  Clearly if a square $(a,b)$ and all squares to the left of it in row $a$ are black, this corresponds to the mapping that takes $a+b-m$ to $a$ and fixes everything else.  

In addition, if $(x_r,y_r)$, $(x_{r+1},y_{r+1})$ are two consecutive steps in the chain rooted at $(x,y)$ and we have applied all $s_i$ coming from squares in the Cauchon diagram between $(x_r,y_r)$ and $(x_{r+1},y_{r+1})$, we know from Lemma~\ref{res:nearest white square in chain} that all squares $(a,b)$ with $x_{r+1} < a \leq m$ and $y_{r+1} < b < y_r$ must also be black. We can apply these elementary transpositions to our set immediately (moving right-to-left, bottom-to-top within this block as usual) since they each commute with all transpositions appearing strictly southeast of them.  See Figure~\ref{fig:commuting permutations} for an illustration of this.

\begin{figure}[h!]

\begin{tikzpicture}
\draw[gray,fill=gray] (0,1) rectangle (3,1.75) (1.5,1.75) rectangle (3,2);
\draw[fill=black] (1.5,2) rectangle (3,3);

\draw rectangle (4,3);
\draw (3,0.75) rectangle +(0.25,0.25) (1.25,1.75) rectangle +(0.25,0.25);

\draw[-] (0,1) -- (3,1) -- (3,3) (0,1.75)--(1.25,1.75) (1.5,2)--(1.5,3);

\node at (-0.3,0.8) {\small $x_r$};
\node at (-0.4,1.9) {\small $x_{r+1}$};
\node[rotate=-90] at (1.33,3.4) {\small $y_{r+1}$};
\node[rotate=-90] at (3.15,3.25) {\small $y_r$};
\end{tikzpicture}

\caption{In this situation, we can apply the black block of permutations northeast of $(x_{r+1},y_{r+1})$ once we've applied all permutations coming from the grey area.}
\label{fig:commuting permutations}
\end{figure}
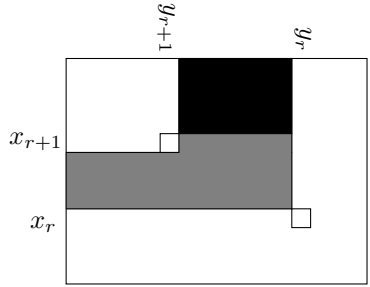

This allows us to make the following definition.  Let $1 \leq r \leq t$, so that $(x_r,y_r)$ is the $r$th step in the chain rooted at $(x,y)$.  By the observation in the previous paragraph, the permutation $v_{x_r,y_r}^{-1}v_{x,y}$ is equal to the subword of $v_{x,y}$ obtained by deleting any $s_i$ coming from squares weakly northwest of $(x_r,y_r)$.  Define $J_r:= v_{x_r,y_r}^{-1}v_{x,y}w_{x,y}^{-1}\interval{1}{m}$; we will prove by induction that
\begin{equation}\label{eq:step for induction in chains proof}J_r = \Big(\interval{1}{x_r}\backslash \{x_1,\dots,x_r\}\Big)\cup \interval{x_r+y_r-m}{y_r-1} \cup \{y_r,\dots, y_1\}\end{equation}
for $1 \leq r \leq t$.

The base step splits into two cases, depending on whether $(x,y)$ is white or black.  If $(x,y)$ is white, then $(x_1,y_1) = (x,y)$ and $J_1 = w_{x,y}^{-1}\interval{1}{m} = \interval{1}{x_1-1} \cup \interval{x_1+y_1-m}{y_1}$ by \eqref{eq:set for all white diagram}.

If $(x,y)$ is black, then $(x_1,y_1)$ is the unique nearest white square weakly northwest of $(x,y)$.  By the observations above, this corresponds to the mapping
\begin{align*}
\interval{x+y-m}{x_1 - 1 + y-m} &\mapsto \interval{x}{x_1-1} & & \text{(any all-black rows)} \\
\interval{x_1+y-m}{y} & \mapsto \interval{x_1 + y_1 - m}{y_1} & & \text{(as in Figure~\ref{fig:commuting permutations})}
\end{align*}
and once again $J_1 = \interval{1}{x_1-1} \cup \interval{x_1+y_1-m}{y_1}$.  In both cases this can be rewritten as $J_1= \Big(\interval{1}{x_1}\backslash\{x_1\}\Big) \cup \interval{x_1+y_1-m}{y_1-1} \cup \{y_1\}$, which agrees with \eqref{eq:step for induction in chains proof} with $r = 1$.

Now assume that $J_r$ is as in \eqref{eq:step for induction in chains proof}, and that $r < t$; we will compute $J_{r+1}$.  The construction of the chain means we don't know anything about the squares in the same row or same column as $(x_r,y_r)$, so the first step is to show that none of the permutations in this row and column can have an effect on $J_r$.  Indeed, consider the permutations to the left of (and in the same row as) $(x_r,y_r)$: they are some subword of $s_{x_r}\dots s_{x_r + y_r - m - 2}$.  Since $J_r \cap \interval{x_r}{x_r+y_r-m-1} = \emptyset$ by the induction hypothesis, this subword has no effect on $J_r$.  Similarly, the permutations above (and in the same column as) $(x_r,y_r)$ form a subword of $s_{y_r-1}\dots s_{x_r + y_r - m}$ and so they simply permute elements in the interval $\interval{x_r+y_r-m}{y_r} \subset J_r$.  So we can ignore all transpositions coming from the same row or column as $(x_r,y_r)$.

If $(x_{r+1},y_{r+1}) = (x_r + 1,y_r-1)$, i.e. the square directly northwest of $(x_r,y_r)$ is white, then the induction step is done since $J_{r+1} = J_r$ and we can write
\begin{align*}
J_{r+1} & = \Big(\interval{1}{x_r}\backslash \{x_1,\dots,x_r\}\Big) \cup \interval{x_r+y_r-m}{y_r-1} \cup \{y_r,\dots, y_1\} \\
& = \Big(\interval{1}{x_r +1}\backslash \{x_1,\dots, x_r +1\}\Big) \cup \interval{(x_r + 1) + (y_r - 1) - m}{y_r-2} \cup \{y_r + 1, y_r,\dots, y_1\} \\
& = \Big(\interval{1}{x_{r+1}}\backslash \{x_1,\dots, x_{r+1}\}\Big) \cup \interval{x_{r+1} + y_{r+1} - m}{y_{r+1}-1} \cup \{y_{r+1},\dots, y_1\}.
\end{align*}
If not, the square $(x_r+1,y_r-1)$ is black.  In this case, the argument from the base step (black square) applies, simply replacing $(x,y)$ with $(x_r,y_r)$ and $(x_1,y_1)$ with $(x_{r+1},y_{r+1})$.  The induction step is proved.

All that remains is to consider what happens when $r=t$.  There are three possible cases: $x_t = m$ (we have reached the top row of the diagram), or $y_t= m+1$ (we have reached the leftmost column of the diagram), or $x_t < m$ and $y_t > m+1$ but all squares strictly northwest of $(x_t,y_t)$ are black.  

In the first two cases, we already have $J_t = v_{x,y}w_{x,y}^{-1}\interval{1}{m}$ and so we just need to check that $J_t$ has the form \eqref{eq:final form of I for chain/permutation result}.  If $x_t = m$, then $\interval{x_t+y_t-m}{y_t-1} = \emptyset$ and the result is clear; similarly, if $y_t = m+1$, then $\interval{x_t+y_t-m}{y_t-1} = \interval{x_t+1}{m}$ and we can write
\begin{align*}J_t &= \Big(\interval{1}{x_t} \backslash \{x_1,\dots, x_t\}\Big) \cup \interval{x_t+1}{m} \cup \{y_t, \dots, y_1\}\\
& = \Big(\interval{1}{m}\backslash \{x_1,\dots, x_t\}\Big)\cup \{y_t,\dots, y_1\}.\end{align*}

Finally, if $x_t < m$ and $y_t > m+1$ but all squares strictly northwest of $(x_t,y_t)$ are black: we have $v_{x,y}w_{x,y}^{-1}\interval{1}{m} = uJ_t$, where
\[u := (s_m\dots s_{m+y_t - 2})\dots(s_{x_t +1}\dots s_{x_t + y_t - m -1})\]
corresponds to the block of black squares strictly northwest of $(x_t,y_t)$. (Any transpositions coming from the same row as $x_t$ or the same column as $y_t$ have no effect, by the same argument as in the induction step, so we can ignore them here.)  This permutation $u$ has the effect of mapping $\interval{x_t+y_t-m}{y_t-1}$ to $\interval{x_t+1}{m}$ and fixing everything else, so we have
\begin{align*}
v_{x,y}w_{x,y}^{-1}\interval{1}{m} &= u\Big( \big(\interval{1}{x_t}\backslash \{x_1,\dots,x_t\}\big)\cup \interval{x_t+y_t-m}{y_t-1} \cup \{y_t,\dots, y_1\}\Big)\\ 
& = \Big(\interval{1}{x_t} \backslash\{x_1,\dots, x_t\}\Big) \cup \interval{x_t+1}{m} \cup \{y_t,\dots, y_1\} \\
& = \Big(\interval{1}{m}\backslash \{x_1,\dots, x_t\}\Big) \cup \{y_t,\dots, y_1\},
\end{align*}
as required.
\end{proof}

Ideally, we would like a description of the chain rooted at the square $(x,y)$ in terms of $v$ and $w$, not $v_{x,y}$ and $w_{x,y}$; otherwise, this is computationally no easier than just working out the chains from the diagrams.  The next proposition translates the results of Proposition~\ref{res:chains in terms of permutations} into an expression involving only $v$ and $w$.  We do this for the case of chains rooted along the southeast border of $C$, since our aim is to study the Grassmann necklace of the diagram; however, it could easily be generalised to any starting square with some careful reindexing.

To simplify the intuition behind the proof, we first prove the result for the case of a rectangular Cauchon diagram with no black squares in the bottom row or rightmost column.

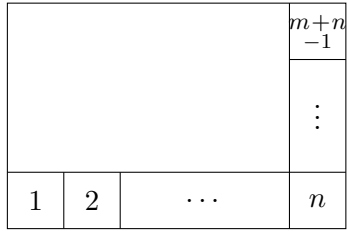
\begin{figure}[h!]
\begin{tikzpicture}[scale=0.75]
\draw (0,0) rectangle (6,4);
\draw (0,1) -- ++(6,0);
\draw (5,0) -- ++(0,4);
\draw (1,0) -- ++ (0,1);
\draw (2,0) -- ++(0,1);
\draw (5,3) -- ++(1,0);
\draw (0.5,0.5) node{$1$} ++ (1,0) node{$2$} ++ (2,0) node{$\dots$} (5.5,0.5) node{$n$} ++ (0,1.62) node{$\vdots$};
\draw (5.5,3.65) node{\footnotesize $m\!+\!n$} ++(0,-0.35) node{\scriptsize $-1$}; 

\end{tikzpicture}
\caption{Numbering the squares along the southeast border of an $m\times n$ Cauchon diagram.}
\label{fig:numbering for southeast border}
\end{figure}

\begin{proposition}\label{res:northwest permutations to permutations, white strip version}
Let $C$ be an $m\times n$ Cauchon diagram with no black squares in the bottom row or rightmost column, and let $v$ be the permutation associated to $C$.  Label boxes in the bottom row and rightmost column of $C$ with the numbers $\interval{1}{m+n-1}$ as in Figure~\ref{fig:numbering for southeast border}, and if $k$ is in box $(x,y)$ then we write $v_k$, $w_k$ for $v_{x,y}$, $w_{x,y}$.  Then
\begin{enumerate}[(i)]
\item If $k \in \interval{1}{n}$, then $v_kw_k^{-1}\interval{1}{m} = \interval{1}{m+k} \backslash  v\interval{1}{k}$.
\item If $k \in \interval{n+1}{m+n-1}$, then $v_kw_k^{-1}\interval{1}{m} = \Big(\interval{1}{m+n}\backslash v\interval{1}{k}\Big) \cup \interval{1}{k-n}$.
\end{enumerate}
\end{proposition}
\begin{proof}

Let $1 \leq k \leq n$. In this case, the permutation $w_k$ is easily seen to be equal to the permutation coming from the all-black $m \times k$ diagram, i.e. we have $w_k^{-1} = (1\ 2\ \dots m+k)^k$.  (We are implicitly identifying $S_{m+k}$ with the subgroup of permutations in $S_{m+n}$ which fix all elements in $\interval{m+k+1}{m+n}$.) It is therefore clear that $w_k^{-1}$ acts on the interval $\interval{1}{m}$ by $w_k^{-1}\interval{1}{m} = \interval{k+1}{k+m}$.  

Meanwhile, since we obtain $v_k$ from $v$ by ignoring any transpositions which occur strictly to the right of box $k$ in the diagram, we can visualise this as deleting from $C$ all columns strictly to the right of $k$ and ``wrapping'' the numbers $\interval{k+1}{k+m}$ up the side of the diagram, then computing the pipe dream of the resulting diagram.  This is illustrated in Figure~\ref{fig: wrapping 1}.

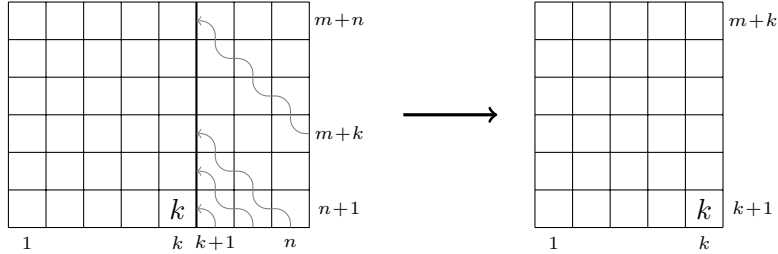
\begin{figure}[h!]
\begin{tikzpicture}
\draw[step=0.5] grid (4,3);
\draw (2.25,0.25) node{$k$};
\draw (0.25,-0.2) node{\tiny $1$} ++(2,0) node{\tiny $k$} ++(0.5,0) node{\tiny $k\!+\!1$} ++(1,0) node{\tiny $n$};
\draw (4.4,0.25) node{\tiny $n\!+\!1$} ++(0,1) node{\tiny $m\!+\!k$} ++(0,1.5) node{\tiny $m\!+\!n$};
\draw[thick] (2.5,0) -- ++(0,3);

\draw[->,rounded corners=0.2cm,gray] (2.75,0) -- ++(0,0.25) -- ++(-0.25,0);
\draw[->,rounded corners=0.2cm,gray] (3.25,0) -- ++(0,0.25) -- ++(-0.5,0) -- ++(0,0.5) -- ++(-0.25,0);
\draw[->,rounded corners=0.2cm,gray] (3.75,0) -- ++(0,0.25) -- ++(-0.5,0) -- ++(0,0.5) -- ++(-0.5,0) -- ++(0,0.5) -- ++(-0.25,0);
\draw[->,rounded corners=0.2cm,gray] (4,1.25) -- ++(-0.25,0) -- ++(0,0.5) -- ++(-0.5,0) -- ++(0,0.5) -- ++(-0.5,0) -- ++(0,0.5) -- ++(-0.25,0);

\draw[->,very thick] (5.25,1.5) -- ++(1.25,0);

\draw[step=0.5cm] (7,0) grid (9.5,3);
\draw (7,0) -- (7,3); 
\draw (7.25,-0.2) node{\tiny 1} ++(2,0) node{\tiny $k$};
\draw (9.9,0.25) node{\tiny $k\!+\!1$} ++(0,2.5) node{\tiny $m\!+\!k$}; 
\draw (9.25,0.25) node{$k$};

\end{tikzpicture}

\caption{``Wrapping'' the numbers $k+1,\dots, k+m$ along the side of the diagram.}
\label{fig: wrapping 1}
\end{figure}

Clearly $v\interval{1}{k} = v_k\interval{1}{k}$ since pipe dreams can't go to the right, and so the first $k$ pipes are unaffected by any deletions we might have made.  We also observe that $v_k\interval{1}{k+m} = \interval{1}{k+m}$ by viewing $v_k$ as a permutation in $S_{m+k}$.  We obtain
\begin{align*}v_kw_k^{-1}\interval{1}{m} & = v_k\interval{k+1}{k+m} \\
& = \interval{1}{k+m} \backslash v_k\interval{1}{k} \\
& = \interval{1}{k+m} \backslash v\interval{1}{k}.
\end{align*}

Case $(ii)$ proceeds in a very similar manner.  Let $k = n+i$ with $1 \leq i \leq m-1$; arguing as in case $(i)$, we see that $w_k^{-1}$ fixes $\interval{1}{i}$ and acts as $j\mapsto j+n$ for $j \in \interval{i+1}{m}$, that $v_k|_{\interval{1}{i}} = id$, and that $v_k|_{\interval{i+1}{m+n}}$ can be visualised by deleting the final $i$ rows from $C$ and wrapping the numbers $\interval{i+1}{i+n}$ along the bottom as illustrated in Figure~\ref{fig:wrapping 2}.

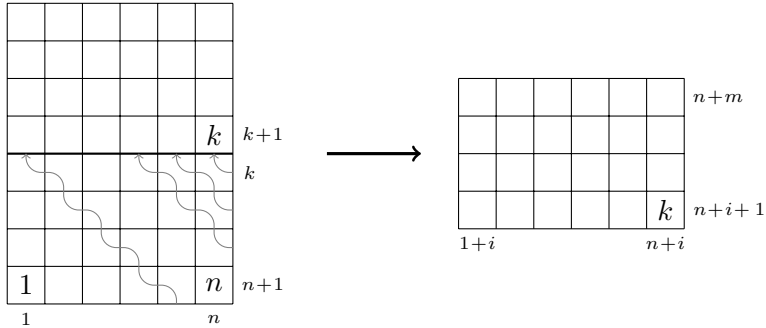
\begin{figure}[h!]

\begin{tikzpicture}
\draw[step=0.5cm] (0,0) grid (3,4);
\draw[thick] (0,2) -- ++(3,0);

\draw (0.25,0.25) node{1} ++(2.5,0) node{$n$}  ++ (0,2) node{$k$};
\draw (0.25,-0.2) node{\tiny $1$} ++(2.5,0) node{\tiny $n$};
\draw (3.4,0.25) node{\tiny $n\!+\!1$} ++(0,1.5) node{\tiny $k \hphantom{+1}$} ++(0,0.5) node{\tiny $k\!+\!1$};

\draw[<-,rounded corners=0.2cm,gray] (2.75,2) -- ++(0,-0.25) -- ++(0.25,0);
\draw[<-,rounded corners=0.2cm,gray] (2.25,2) -- ++(0,-0.25) -- ++(0.5,0) -- ++(0,-0.5) -- ++(0.25,0);
\draw[<-,rounded corners=0.2cm,gray] (1.75,2) -- ++(0,-0.25) -- ++(0.5,0) -- ++(0,-0.5) -- ++(0.5,0) -- ++(0,-0.5) -- ++(0.25,0);

\draw[<-,rounded corners=0.2cm,gray] (0.25,2) -- ++(0,-0.25) -- ++(0.5,0) -- ++(0,-0.5) -- ++(0.5,0) -- ++(0,-0.5) -- ++(0.5,0) -- ++(0,-0.5) -- ++(0.5,0) -- ++(0,-0.25);

\draw[->,very thick] (4.25,2) -- ++(1.25,0);

\draw[step=0.5cm] (6,1) grid (9,3);
\draw (6,1) -- +(0,2) (6,1) -- +(3,0);

\draw (8.75,1.25) node{$k$};
\draw (6.25,0.8) node{\tiny $1\!+\!i$} ++(2.5,0) node{\tiny $n\!+\!i$};
\draw (9.6,1.25) node{\tiny $n\!+\!i\!+1$} ++(-0.15,1.5) node{\tiny $n\!+\!m$};

\end{tikzpicture}
\caption{Wrapping the numbers $i+1, \dots, i+n$ along the bottom of the diagram.}
\label{fig:wrapping 2}
\end{figure}

Once again, we note that $v\interval{1}{m+n} = \interval{1}{m+n}$ and $v_k\interval{n+i+1}{m+n} = v\interval{n+i+1}{m+n}$, so that
\begin{align*}v_k\interval{n+i+1}{m+n} & = v\interval{n+i+1}{m+n}\\
& = \interval{1}{m+n}\backslash v\interval{1}{n+i}.
\end{align*}
Combining everything, we have
\begin{align*} v_kw_k^{-1}\interval{1}{m} &= v_k\Big(\interval{1}{i}\cup \interval{n+i+1}{m+n}\Big) \\
& = \interval{1}{i}\cup \Big(\interval{1}{m+n}\backslash v\interval{1}{n+i}\Big).\end{align*}
Since $k=n+i$, this is exactly the set promised by the statement of the theorem.
\end{proof}

\begin{corollary}\label{res:arbitrary rectangle cauchon diagram, chain rooted at to permutation}
Let $C$ be an $m\times n$ Cauchon diagram, and number the boxes in the bottom row and rightmost column with the numbers $\interval{1}{m+n-1}$ as in Figure~\ref{fig:numbering for southeast border}.  Let $(x_1,y_1), \dots, (x_t,y_t)$ be the chain rooted in the box labelled $k$.  If $k \in \interval{1}{n}$, then
\begin{equation*}
\Big(\interval{1}{m} \backslash \{x_1,\dots,x_t\}\Big) \cup\{y_t,\dots, y_1\} = \interval{1}{m+k} \backslash  v\interval{1}{k},
\end{equation*}
while if $k \in \interval{n+1}{m+n-1}$ then
\begin{equation*}
\Big(\interval{1}{m} \backslash \{x_1,\dots,x_t\}\Big) \cup\{y_t,\dots, y_1\} = \Big(\interval{1}{m+n} \backslash  v\interval{1}{k}\Big) \cup \interval{1}{k-n}.
\end{equation*}
\end{corollary}
\begin{proof}\ 
Let $C'$ be the $(m+1) \times n$ diagram obtained from $C$ by adding an extra row of white squares underneath the bottom row, and let $C''$ be the $(m+1)\times (n+1)$ diagram obtained from $C'$ by adding an extra column of white squares on the right hand side.  Clearly $C''$ is a Cauchon diagram if and only if $C$ is.  We will number the rows of $C''$ with the numbers $\interval{0}{m}$ from bottom to top, and the columns of $C''$ with the numbers $\interval{m+1}{m+n+1}$ from left to right.  Finally, we number this strip of white squares in the bottom row and column of $C''$ with $0',1',\dots, (m+n)'$ running from bottom left to top right according to the same convention as Figure~\ref{fig:numbering for southeast border}.

Now $C''$ satisfies the conditions of Proposition~\ref{res:northwest permutations to permutations, white strip version}, and clearly the chain rooted at box $i$ of $C$ is exactly the chain rooted at box $i'$ of $C''$ with the first term removed.  The result now follows immediately from Propositions~\ref{res:chains in terms of permutations} and \ref{res:northwest permutations to permutations, white strip version}.
\end{proof}

We can now state and prove our second main theorem, which relates the minors coming from a Grassmann necklace to the minors in Theorem~\ref{res:sian-milen thm}.

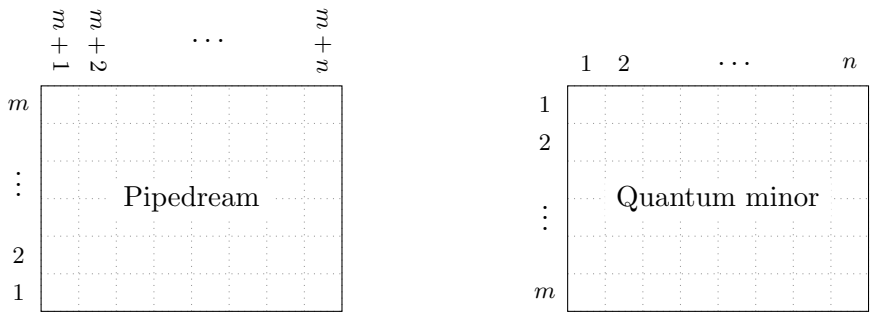
\begin{figure}[h!]

\begin{tikzpicture}

\draw[light-gray,dotted,step=0.5] grid (4,3);
\draw rectangle (4,3);
\draw (-0.3,0.25) node{\footnotesize 1};
\draw (-0.3,0.75) node{\footnotesize 2};
\draw (-0.3,1.8) node{$\vdots$};
\draw (-0.3,2.75) node{\footnotesize $m$};
\node[rotate=-90] at (0.25,3.6) {\footnotesize $m+1$};
\node[rotate=-90] at (0.75,3.6) {\footnotesize $m+2$};
\draw (2.25,3.6) node{$\dots$};
\node[rotate=-90] at (3.75,3.6) {\footnotesize $m+n$};


\draw[light-gray,dotted,step=0.5] (7,0) grid (11,3);
\draw (7,0) rectangle (11,3);
\draw (6.7,0.25) node{\footnotesize $m$};
\draw (6.7,1.35) node{$\vdots$};
\draw (6.7,2.25) node{\footnotesize $2$};
\draw (6.7,2.75) node{\footnotesize $1$};
\node at (7.25,3.3) {\footnotesize $1$};
\node at (7.75,3.3) {\footnotesize $2$};
\draw (9.25,3.3) node{$\dots$};
\node at (10.75,3.3) {\footnotesize $n$};

\draw[fill=white,white] (0.9,1.25) rectangle ++(2.15,0.5);
\draw[fill=white,white] (7.5,1.25) rectangle ++(3,0.5);

\draw (2,1.5) node{Pipedream};
\draw (9,1.5) node{Quantum minor};


\end{tikzpicture}
\caption{Translating between row/column numbering conventions for pipedreams and quantum minors.}
\label{fig:translate minors pipedreams}
\end{figure}

\begin{theorem}\label{res:yakimov's minors come from the chains rooted at squares}
Let $C$ be an $m\times n$ rectangular Cauchon diagram with associated permutation $v$. We label the squares in the bottom row and rightmost column of $C$ with the numbers $\interval{1}{m+n-1}$ as in Figure~\ref{fig:numbering for southeast border}, and define $I_k = v\interval{1}{k}$.  Then the quantum minor defined by the chain rooted at box $k$ is exactly
\[\Big[w_m^o\big(p_1(I_k)\big)\Big|\interval{m+1}{m+k}\backslash p_2(I_k) - m\Big]_q\]
if $k \in \interval{1}{n}$, and
\[\Big[w_m^o\big(p_1(I_k)\backslash \interval{1}{k-n}\big)\Big| \interval{m+1}{m+n} \backslash p_2(I_k) - m\Big]_q\]
if $k \in \interval{n+1}{m+n-1}$.  Hence the separating Ore set for $K$ constructed in Theorem~\ref{res:GN constructs the Ore set we need} is equal to the separating Ore set for $K$ in Theorem~\ref{res:sian-milen thm}.
\end{theorem}

\begin{proof}
As always, this splits into two cases depending on whether $k \leq n$ or $k \geq n+1$.  

Suppose $k \in \interval{1}{n}$.  We know that $\{x_1,\dots, x_t\} \subset \interval{1}{m}$ and $\{y_1,\dots, y_t\} \subset \interval{m+1}{m+n}$ by the definition of the chain construction, and so by Corollary~\ref{res:arbitrary rectangle cauchon diagram, chain rooted at to permutation}, we have 
\begin{align*}\Big(\interval{1}{m} \backslash \{x_1,\dots,x_t\}\Big) \cup\{y_t,\dots, y_1\} &= \interval{1}{m+k} \backslash  I_k\\
& = \Big(\interval{1}{m} \backslash p_1(I_k)\Big) \cup \Big(\interval{m+1}{m+k}\backslash p_2(I_k)\Big).
\end{align*}
It follows that $\{x_1,\dots, x_t\} = p_1(I_k)$ and $\{y_1,\dots, y_k\} = \interval{m+1}{m+k}\backslash p_2(I_k)$.  All that remains is to translate the row/column numbering from the pipedream convention to the standard numbering for quantum minors, as illustrated in Figure~\ref{fig:translate minors pipedreams}: applying $w_m^o$ to the row set and subtract $m$ from each entry in the column set, we obtain exactly the minor promised in the statement of the theorem.

Now let $k \in \interval{n+1}{m+n-1}$.  By Corollary~\ref{res:arbitrary rectangle cauchon diagram, chain rooted at to permutation}, we have
\begin{equation}\label{eq:computing minor for k >= n+1}\begin{aligned}
\Big(\interval{1}{m} \backslash \{x_1,\dots,x_t\}\Big) &\cup\{y_t,\dots, y_1\} \\
& = \Big(\interval{1}{m+n} \backslash  I_k\Big) \cup \interval{1}{k-n}\\
& = \interval{1}{m} \backslash p_1(I_k) \cup \interval{m+1}{m+n}\backslash p_2(I_k) \cup \interval{1}{k-n}.
\end{aligned}\end{equation}
Since $v$ is a restricted permutation and $k \geq n+1$, we must have $\interval{1}{k-n} \subseteq p_1(I_k)$ and hence $\big(\interval{1}{m} \backslash p_1(I_k)\big) \cap \interval{1}{k-n}= \emptyset$.  It therefore follows from \eqref{eq:computing minor for k >= n+1} that $\{x_1,\dots,x_t\} = p_1(I_k) \backslash \interval{1}{k-n}$, and clearly $\interval{m+1}{m+n} \backslash p_2(I_k) = \{y_1,\dots, y_t\}$.  The result now follows as above.
\end{proof}
This also lets us express the Grassmann necklace of a cell in terms of the restricted permutation.
\begin{theorem}\label{res:gn in terms of permutation}
Let $C$ be a rectangular $m\times n$ Cauchon diagram, and $v$ its associated restricted permutation.  Then the Grassmann necklace of $C$ is $\II_C= (I_1, \dots, I_{m+n})$, where 
\[I_{m+n-k+1} = \left\{\begin{array}{cl} \Big(\interval{1}{m}\cup \interval{m+n-k+1}{m+n}\Big) \big\backslash w_m^ow_{m,n}^o v \interval{1}{k}, &  k \in \interval{1}{n}; \\
\Big(\interval{1}{m+n} \backslash w_m^ow_{m,n}^o v\interval{1}{k} \Big) \cup \interval{m+n-k+1}{m}, & k \in \interval{n+1}{m+n}.
\end{array}\right.\]
\end{theorem}
\begin{proof}
Recall from \eqref{eq:translate minors and plucker coordinates} that the translation from quantum minors to Pl\"ucker coordinates is given by 
\[[S|T]_q \mapsto (\interval{1}{m} \backslash S) \cup (w_n^o(T) + m).\]
Applying this to the formulas in Theorem~\ref{res:yakimov's minors come from the chains rooted at squares}, we get
\[\interval{1}{m} \backslash w_m^o\big(p_1(I_k)\big) \cup  w_{m,n}^o\big(\interval{m+1}{m+k} \backslash p_2(I_k)\big),\]
when $k \in \interval{1}{n}$, and
\[\interval{1}{m}\backslash w_m^o\big(p_1(I_k)\big) \cup \interval{m+n-k+1}{m} \cup w_{m,n}^o\big(\interval{m+1}{m+n} \backslash p_2(I_k)\big),\]
when $k \in \interval{n+1}{m+n}$.  After simplifying and writing $I_k = v\interval{1}{k}$, these clearly agree with the sets in the statement of the theorem.
\end{proof}

\section{Application to Special Cases}\label{s:special cases}

Finally, we can make a few elementary observations about what happens in certain special cases: in \S\ref{ss:J=K} we describe what happens when $J = K$, and in \S\ref{ss:ext to gln sln} we extend Theorem~\ref{res:we have constructed ore sets yay} to the quantized coordinate rings of $GL_n$ and $SL_n$.

\subsection{The case where $J = K$}\label{ss:J=K}

Our original motivation was constructing denominator sets $E_{JK}$ for pairs of $\HH$-primes $J \subsetneq K$ in order to study the algebras defined in \eqref{eq:first def of ZJK}; however, since $E_K \cap K = \emptyset$ we can also consider the image of $E_K$ in $\OO_q(M_{m,n})/K$.

In this case, the set $E_{KK}$ (following the notation in Theorem~\ref{res:we have constructed ore sets yay} with $J = K$) is a finitely generated denominator set of $\HH$-eigenvectors which satisfies the following condition:
\begin{itemize}\item For all $L \in \HH$-$spec(\OO_q(M_{m,n}))$ with $K \subsetneq L$, $E_{KK} \cap (L/K) \neq \emptyset$.\end{itemize}
In other words, the localization $A_K:=(\OO_q(M_{m,n})/K)[E_{KK}^{-1}]$ is $\HH$-simple: it has no non-trivial $\HH$-primes.  The centre $Z_{KK} = Z(A_K)$ is therefore exactly the commutative Laurent polynomial ring appearing in the statement of the Stratification Theorem \cite[II.2.13]{BGbook}, so we can also use the results of Theorem~\ref{res:we have constructed ore sets yay} as a starting point for studying the topological structure of the individual strata $spec_K(\OO_q(M_{m,n}))$.

Other methods for constructing denominator sets leading to $\HH$-simple localizations already exist, e.g. Launois-Lenagan \cite{LL1} and Yakimov \cite{Yakimov2}.  Our sets are smaller in general than those of Launois and Lenagan; however, their aim was to construct not just $\HH$-simple localizations but ones with an especially nice structure, which required a larger denominator set.  By Theorem~\ref{res:yakimov's minors come from the chains rooted at squares}, our denominator sets $E_{KK}$ recover exactly the denominator sets of \cite[Theorem~3.1]{Yakimov2} in the case $R_q[G] = \OO_q(M_{m,n})$.

\subsection{Extension to $\OO_q(GL_n)$ and $\OO_q(SL_n)$}\label{ss:ext to gln sln}

Throughout this section, we set $m = n \geq 2$.

It is well known that the \textit{quantum determinant}, i.e. the $n\times n$ quantum minor $D_q := [1, \dots, n| 1, \dots, n]_q$, is central in $\OO_q(M_{n,n})$.  We can then define the quantized coordinate rings of $GL_n$ and $SL_n$ as follows:
\[\OO_q(GL_n):= \OO_q(M_{n,n})[D_q^{-1}], \qquad \OO_q(SL_n):= \OO_q(M_{n,n})/(D_q-1).\]
The action of $\HH = (\KK^{\times})^{2n}$ in \eqref{eq:H action def} extends naturally to a rational $\HH$-action on $\OO_q(GL_n)$.  For $\OO_q(SL_n)$, we instead need to use the smaller torus
\begin{equation}\label{eq:def of HH',action on SLn}\HH' := \{(\alpha_1,\dots, \alpha_n,\beta_1,\dots, \beta_n) \in (\KK^{\times})^{2n} : \alpha_1\alpha_2\dots\beta_n = 1\} \subset \HH.\end{equation}
The definition in \eqref{eq:H action def} induces a rational action of $\HH'$ on $\OO_q(SL_n)$ \cite[II.1.16]{BGbook}.

Since $\OO_q(GL_n)$ is just a localization of $\OO_q(M_{n,n})$, it follows from standard localization theory that the $\HH$-primes of $\OO_q(GL_n)$ are exactly
\begin{equation}\label{eq:h primes in gln}\HH\text{-}spec(\OO_q(GL_n)) = \{J[D_q^{-1}]: J \in \HH\text{-}spec(\OO_q(M_{n,n})),\ D_q \not\in J\}.\end{equation}
The following result now follows easily from Theorem~\ref{res:we have constructed ore sets yay}.
\begin{corollary}\label{res:gln}
Let $K \in \HH$-$spec(\OO_q(GL_n))$ with $q \in \KK^{\times}$ not a root of unity, and identify $K$ with the corresponding $\HH$-prime in $\OO_q(M_{n,n})$.  Then the set $E_K$ from Theorem~\ref{res:GN constructs the Ore set we need} lifted to $\OO_q(GL_n)$ is also a separating Ore set for $K$ in $\OO_q(GL_n)$.
\end{corollary}
\begin{proof}
The Cauchon diagrams of $\HH$-primes in $\OO_q(GL_n)$ are characterised as exactly those $n \times n$ Cauchon diagrams with no black squares on the main diagonal.  Therefore $D_q \in E_K$ and the result follows.
\end{proof}
\begin{remark}
In \cite[Example~4.3]{BGconjecture}, Brown and Goodearl constructed denominator sets $E_{JK}$ for all pairs of $\HH$-primes $J \subseteq K$ in the specific case of $\OO_q(GL_2)$.  Their sets are slightly smaller than those predicted by Corollary~\ref{res:gln}: for example, when $J = 0$ and $K = \langle b \rangle$, \cite[Example~4.3]{BGconjecture} lists the generating set $\{c,D_q\}$ while the corresponding set from Corollary~\ref{res:gln} is $\{a,c,D_q\}$.  (Here we are using the standard notation $a,b,c,d$ for the generators of $\OO_q(GL_2)$.)

If minimal generating sets are required when computing examples, this redundancy in the sets from Corollary~\ref{res:gln} can be eliminated by removing any minors whose chain begins at a square on the main diagonal of the Cauchon diagram, i.e. such that $(x_1,y_1) = (i,2n-i+1)$ for some $i \in \interval{1}{n}$.
\end{remark}

For $\OO_q(SL_n)$ we require one extra step, since we need to work with the action of $\HH'$ from \eqref{eq:def of HH',action on SLn} rather than the original $\HH$.  The following result from \cite{GL1} provides the necessary translation between the two.
\begin{proposition}\label{res:map h primes in gln to sln} \cite[Proposition~2.5]{GL1} 
Let $\pi: \OO_q(M_{n,n}) \longrightarrow \OO_q(SL_n)$ denote the natural quotient map, and extend this to a map $\OO_q(GL_n) \longrightarrow \OO_q(SL_n)$ also denoted by $\pi$.  Then the map of sets 
\[\HH\text{-}spec(\OO_q(GL_n)) \longrightarrow \HH'\text{-}spec(\OO_q(SL_n)) : P \mapsto \pi(P),\]
is a bijection.
\end{proposition}

The result for $\OO_q(SL_n)$ now follows immediately from Corollary~\ref{res:gln} and Proposition~\ref{res:map h primes in gln to sln}.
\begin{corollary}\label{res:sln}
Let $K \in \HH'$-$spec(\OO_q(SL_n))$ with $q \in \KK^{\times}$ not a root of unity, and let $E_K$ be the separating Ore set of the corresponding $\HH$-prime in $\OO_q(GL_n)$.  Then $\overline{E}_K := \pi(E_K)$ is a separating Ore set for $K$ in $\OO_q(SL_n)$.
\end{corollary}
In \cite{Fryer}, the second author constructed denominator sets $E_{JK}$ for all pairs of $\HH'$-primes $J \subsetneq K$ in $\OO_q(SL_3)$ in order to describe the topological structure of $spec(\OO_q(SL_3))$.  These were constructed using lacunary sequences from white squares rather than chains rooted at squares along the southeast border of the Cauchon diagram, so the sets in \cite[Table 3]{Fryer} are different from those constructed in Corollary~\ref{res:sln}.  As Example~\ref{ex:lacunary sequences fail in sl4} indicates, we expect that Corollary~\ref{res:sln} will be the better method for generalising the results of \cite{Fryer} to larger algebras.

\subsection*{Acknowledgements}The first author was supported by a Marie Curie Visiting Research Fellowship which was held at the University of Kent. The second author would like to thank the University of Kent for their hospitality in summer 2015, where the initial work for this paper was carried out.  We also thank St\'ephane Launois, Tom Lenagan, Robert Marsh, and Milen Yakimov for many helpful discussions and comments. 

\bibliographystyle{plain}

\end{document}